\documentclass[leqno]{article}
\usepackage{amsmath}
\usepackage[latin1]{inputenc}
\usepackage{amssymb}
 \usepackage[mathscr]{euscript}
\usepackage{xcolor}
\usepackage{cite}
\usepackage{hyperref}
\usepackage{amsthm} 
\usepackage{enumerate}                                        
\usepackage{url}
\usepackage{graphicx}
\usepackage{epsfig}
\newtheorem{theorem}{\bf Theorem}

\newtheorem{corollary}[theorem]{\bf Corollary}
\newtheorem{lemma}[theorem]{\bf Lemma}

\newtheorem{claim}{\bf Claim}
\newtheorem{proposition}[theorem]{\bf Proposition}
\newtheorem{definition}[theorem]{\bf Definition}

\newtheorem{remark}[theorem]{\bf Remark}

\numberwithin{equation}{section}
\numberwithin{theorem}{section}
\numberwithin{figure}{section}

\def\R{\mathbb{R}}
\def\L{\mathbb{L}}
\def\R{\mathbb{R}}

\def\phie3{[\varphi, \vec{e}_3]}
\def\*phie3{[\varphi^*, \vec{e}_3]}
\def\ll{\langle\langle}
\def\gg{\rangle\rangle}
\def\cxi{\bar{\xi}}

\begin{document}
\renewcommand{\thefootnote}{}
\renewcommand{\thefootnote}{}
\footnotetext{Research partially supported by Ministerio de Ciencia e Innovaci\'on Grants No: PID2020-118137GB-I00/AEI/10.13039/501100011033, PID2021-126217NB-100 and the `Maria de Maeztu'' Excellence
Unit IMAG, reference CEX2020-001105-M, funded by
MCIN/AEI/10.13039/501100011033, and CNPq - Conselho Nacional de Desenvolvimento Cient\'ifico e Tecnol\'ogico, Brazil, grant number 402589/2022-0.}

\title{Stationary surfaces of height-dependent weighted area functionals in $\mathbb{R}^3$ and $\mathbb{L}^3$} 
\author{Antonio Mart\'inez$^1$, A. L. Mart\'inez-Trivi\~no$^2$, J. P. dos Santos$^3$}
\vspace{.1in}
\date{}
\maketitle
{
\noindent $^1$Departamento de Geometr\'\i a y Topolog\'\i a, Universidad de Granada, E-18071 Granada, Spain\\ 
\noindent $^2$ Departamento de Matem\'atica Aplicada, Universidad de C\'ordoba, E-14071 C\'ordoba, Spain \\
\noindent $^3$ Departamento de Matem\'atica, Universidade de Bras\'\i lia, Bras\'\i lia-DF, Brazil \\ \\
e-mails: amartine@ugr.es, almartinez@uco.es, joaopsantos@unb.br}
\\
\begin{abstract} We describe a general correspondence between weighted minimal surfaces in $\R^3$  and weighted maximal surfaces with some admissible singularities in $\mathbb{L}^3$, for a class of functions $\varphi$ which provides the corresponding weight.  For these families of surfaces, we provide a Weierstrass representation  when $\dot{\varphi}\neq 0$ and analyze in detail the asymptotic behavior of both such a weighted maximal surface around its singular set and its corresponding weighted minimal immersion  around the nodal set of its angle function, establishing criteria that allow us to easily determine the type of singularity and classify the associated moduli spaces.  \\
\end{abstract}

\section{Introduction}
In this paper we consider  the class of stationary surfaces for the  following weighted area functional 
\begin{equation}
\label{area}
\mathcal{A}^{\overline{\varphi}}(\Sigma)=\int_{\Sigma}e^{\epsilon\overline{\varphi}}\, d \Sigma, \qquad \epsilon\in\{-1,1\},
\end{equation}
on isometric immersions of  Riemannian  surfaces $\Sigma$ in a domain $\mathfrak{D}^3$ of $ \mathbb{R}^3$ when $\epsilon=1$ (or $\mathbb{L}^3$ when $\epsilon=-1$) where $\overline{\varphi}$ is the restriction on $\Sigma$ of a smooth function $\varphi:\mathfrak{D}^3\rightarrow \R$ depending only on the coordinate of  $\mathfrak{D}^3$  in the direction of $\vec{e}_3= (0, 0, 1)$, $\varphi(q)=\varphi(\langle q,\vec{e}_3\rangle$), $q\in \mathfrak{D}^3$ and 
$d\Sigma$ denotes the volume element induced on $\Sigma$ by the  Euclidean (or Lorentzian) metric $\langle\cdot,\cdot\rangle:= dx^2+dy^2+dz^2$ ($\ll\cdot,\cdot\gg:= dx^2+dy^2-dz^2$). 

The Euler-Lagrange equation of \eqref{area} is given in terms of the mean curvature vector $\textbf{H}$ of $\Sigma$ as follows
\begin{equation}
\label{meancurvature}
\textbf{H}=(\epsilon\overline{\nabla}( \varphi))^{\perp} = \dot{\varphi} \ \vec{e}_3^{\,\perp},
\end{equation}
where $\perp$ is the projection on the normal bundle, $\overline{\nabla}$ stands the usual gradient operator in $\R^3$  (or
$\mathbb{L}^3$) and  by dot we denote derivative of $\varphi$ with respect to the height function.
\\

Any critical point of \eqref{area}  in $\mathbb{R}^{3}$ (or $\mathbb{L}^3$)   will be called  {\sl $[\varphi, \vec{e}_3]$-minimal } (or {\sl spacelike $[\varphi, \vec{e}_3]$-maximal}) surface. They can also be characterized, respectively, as minimal (spacelike maximal) surfaces in $\R^3$ endowed with the conformally Euclidean (Lorentzian) changed metric, $\langle \cdot,\cdot\rangle^\varphi = e^{\varphi}\langle \cdot,\cdot\rangle$, ($\ll \cdot,\cdot\gg^\varphi = e^{-\varphi}\ll \cdot,\cdot\gg$), see \cite{Ilm94}.

Although minimal (maximal) immersions are the best known examples in this large family of surfaces, during  the last few years a renewed interest has arisen in their study and  important advances have been obtained for some particular functions $\varphi$, namely,
\begin{itemize}
\item when $\varphi$ is just the height function, $\varphi(p)=\langle p,\vec{e}_3\rangle$, we have the translating solitons, that is, surfaces $\Sigma$ such that $$ t \rightarrow \Sigma + t \vec{e}_3 $$
is a mean curvature flow, i.e. the normal component of the  velocity at  each point is equal to the mean curvature at that point (see \cite{CSS,QD,HMW, HIMW,HIMW2,MSHS1}).
\item when  $\varphi(p) =\alpha \log \langle p,\vec{e}_3\rangle$, $\alpha$=const, we have singular $\alpha$-minimal ($\alpha$-maximal) surface.  In the Euclidean case and when  $\alpha=1$,  $\Sigma$ describes the shape of a ``hanging roof'', i.e. a heavy surface in a gravitational field. Such surfaces are of importance for the construction of perfect domes (see \cite{D,DH,Rafa1,Rafa2,Rafa3,Rafa4,MMT-JMAA}).
\end{itemize} 
The studies carried out to date have addressed essentially two lines of research. One of them allows us to understand the difference we find between the geometry of a minimal (maximal) surface with the cases in which $\dot{\varphi}\neq 0$, and the other one refers to the need to compensate the lack of complete examples in $\mathbb{L}^3$ (see \cite{Ca,CQ}) with the study of $[\varphi,\vec{e}_3]$-maximal surfaces  with some admissible singularities (see \cite{ER,FLR,FSUY,Koba,UY-HM}).

In the present work we follow this last direction and 
introduce the notion of $[\varphi^*,\vec{e}_3]$-maxface which is nothing other than a  spacelike $[\varphi^*,\vec{e}_3]$-maximal surface in $\mathbb{L}^3$ with some admissible singularities but with a global well-defined Gauss map and  with a partner  in $\R^3$ 
which is a $ [ \varphi,\vec{e}_3]$-minimal surface for some  $\varphi$ determined  by $\varphi^*$ up to a constant. The partner is obtained through a correspondence we will introduce, which generalizes the well-known Calabi's correspondence between minimal surfaces in $\mathbb{R}^3$ and maximal surfaces in $\mathbb{L}^3$ (see \cite{Ca,MMT-NL}).

Our goal is to analyze in detail the asymptotic behavior of both the $[\varphi^*,\vec{e}_3]$-maxface around its singular set and its corresponding  $ [ \varphi,\vec{e}_3]$-minimal immersion  around the nodal set of its angle function, establishing criteria that allow us to easily determine the type of singularity and classifying the associated moduli spaces.

Among our main results, we begin by highlighting that when $\varphi$ is real analytic and $\dot{\varphi}\neq0$,  $ [ \varphi,\vec{e}_3]$-minimal immersions in $\R^3$ are  uniquely determined around a curve where the mean curvature vanishes and the Gauss curvature is non-zero (see Corollary  \ref{nodal-unicity}). 

We also highlight the following classification results
\setcounter{theorem}{0}
\renewcommand{\thetheorem}{\Alph{theorem}}
\begin{theorem}\label{Th-A}Let $I^*\subseteq \R$  be an open interval,  $\varphi^*\in {\cal C}^\omega(I^*)$ and $p^*\in \mathbb{L}^3$ with $-\ll p^*,\vec{e}_3\gg \in I^*$. 
Consider ${\cal M}_1$ the class of spacelike $\*phie3$-maxfaces $\psi^*: \Sigma\rightarrow\mathbb{L}^3$ such that $p\in\Sigma$ is a non-degenerate singular point of $\psi^*$ with $\psi^*(p)=p^*$ and take 
${\cal M}_2(\varepsilon)=\{ (A^*,B^*)\in {\cal C}^\omega(]-\varepsilon,\varepsilon[)\times{\cal C}^\omega(]-\varepsilon,\varepsilon[) | \ \ A^{*2 }+ B^{*2} \neq 0\}$. Then, for $\varepsilon$ small enough,  there exists an explicitly constructed  bijective correspondence $ {\cal M}_1\longleftrightarrow {\cal M}_2(\varepsilon)$.\end{theorem}

\begin{theorem}\label{Th-B}Let $\varphi^*\in {\cal C}^\omega(I^*)$ and $p^*=(p^*_1,p^*_2,p^*_3)\in\mathbb{L}^3$ a fixed point such that $p^*_3\in I^*$. Let ${\cal M}$ be  the class of all spacelike $\*phie3$-maximal vertical graphs in $\mathbb{L}^3$, with upwards-pointing normal, having $p^*$ as a non-removable isolated singularity and whose Gaussian curvature does not vanish near $p^*$ (we identify two graphs in ${\cal M}$ if they overlap on an open set containing $p^*$). Then, there is a one-to-one correspondence between ${\cal M}$ and the class $\mathscr{P}$ of analytic regular planar Jordan curves bounding a star-shaped domain at the origin.
\end{theorem}
About these results, we should point out that real analyticity of $\varphi$  is used  in order to ensure the previous existence and uniqueness properties. Furthermore, in Theorem A, the type of singularity is completely determined by the zeros of the functions $A^*$ and $B^*$ (see Proposition \ref{type-singularity}).  
Although the mean curvature goes to infinity on the singular set of a $\*phie3$-maxface, the proof of theorems A and B have been inspired by some ideas developed in \cite{GJM-IHP} about spacelike maximal graphs with prescribed analytic mean curvature and by the correspondence between $\phie3$-minimal and $\*phie3$-maxfaces described in Section \ref{sec-2}, which has allowed us to use the existence of a suitable conformal parametrization around the singularity.

The paper is organized as follows. In Section \ref{sec-2}, we  extend the correspondence described in \cite{MMT-NL} between $\phie3$-minimal graphs of $\mathbb{R}^3$ and  $\*phie3$-maximal spacelike graphs of $\mathbb{L}^3$ to   arbitrary  $\phie3$-minimal surfaces and  $\*phie3$-maxfaces. 

In Section  \ref{sec-3}, we show that if $\dot{\varphi}\neq 0$, then  $\phie3$-minimal surfaces and $[\varphi^*,\vec{e}_3]$-maxfaces can be represented in terms of its Gauss map.

In Section  \ref{sec-4},  we provide some criteria on admissible singularities for a $[\varphi^*,\vec{e}_3]$-maxface in terms of geometric data of its $[\varphi,\vec{e}_3]$-minimal partner. We also introduce a  canonical conformal parametrization around any non-degenerate singularity that allows us, on the one hand, to show that if $\varphi$ is analytic and $\dot{\varphi}\neq0$, the  nodal set  for the mean curvature of a $[\varphi,\vec{e}_3]$--minimal surface  determines, uniquely, the surface at the points where the Gaussian curvature is not zero and, on the other hand, to show that the asymptotic behavior of a $[\varphi^*,\vec{e}_3]$-maxface around a non-degenerate singularity depends on two real analytic functions $A^*$ and $B^*$, that allow us to define the correspondence $ {\cal M}_1\longleftrightarrow {\cal M}_2(\varepsilon)$ in  Theorem \ref{Th-A}.

Finally, Section  \ref{sec-5} is devoted  to studying isolated singularities of $[\varphi^*,\vec{e}_3]$-maximal graphs in $\mathbb{L}^3 $ whose Gaussian curvature does not vanish around the singularity and  prove  Theorem \ref{Th-B}.
\setcounter{theorem}{0}
\renewcommand{\thetheorem}{\arabic{theorem}}
\section{The Calabi's pair}\label{sec-2}

Our objective in this section is to extend the correspondence described in \cite{MMT-NL} between $\phie3$-minimal graphs of $\mathbb{R}^3$ and  $\*phie3$-maximal spacelike graphs of $\mathbb{L}^3$ to   arbitrary  $\phie3$-minimal surfaces.
\

Let $\psi=(x,y,z) :\Sigma \rightarrow \mathbb{R}^3$ be an immersion  from a simply-connected surface $\Sigma$ into $\R^3$, we can regard naturally $\Sigma $ as a Riemann surface  with the conformal structure determined by its induced metric $ds^2$. Consider $\xi=u+iv$  a conformal parameter on $\Sigma$ so that $ds^2 = E |d\xi|^2$ for some positive function $E$. Then, it is not difficult to see that $\psi$ is  a $\phie3$-minimal immersion if and only if $\psi$ satisfies in terms of $\xi$ the following elliptic system:
\begin{equation} \label{eq:system-psi}
\begin{array}{rcl}
2 x_{\xi\cxi} + \dot{\varphi} (z_{\xi} x_{\cxi} + z_{\cxi} x_{\xi}) &=& 0, \vspace{.05in}\\
2 y_{\xi\cxi} + \dot{\varphi} (z_{\xi} y_{\cxi} + z_{\cxi} y_{\xi}) &=& 0,\vspace{.05in} \\
2 z_{\xi\cxi}-\dot{\varphi} (|x_{\xi}|^2+|y_{\xi}|^2 - |z_{\xi}|^2)&=& 0.
\end{array}
\end{equation}

We will show that there exists a differentiable map $\psi^{*} : \Sigma  \rightarrow \mathbb{R}^3$  with  $\psi^{*} = (x^{*}, y^{*}, z^{*})$ satisfying:
\begin{equation} \label{CR-*}
\begin{array}{rcl}
x^{*}_{\xi} &=& -i e^{\varphi} y_{\xi},  \\
y^{*}_{\xi} &=&  i e^{\varphi} x_{\xi}, \\
z^{*}_{\xi} &=& e^{\varphi} z_{\xi}.
\end{array}
\end{equation}
\begin{proposition} \label{integrability-sys}
Let $\psi=(x,y,z) :\Sigma\rightarrow \mathbb{R}^3$  be a $\phie3$-minimal immersion and $\xi=u+iv$ is a conformal parameter of $\Sigma$. Then the system \eqref{CR-*} is integrable.
\end{proposition}
\begin{proof}
Let us show that the system \eqref{eq:system-psi} implies $x^{*}_{\xi\cxi} = x^{*}_{\cxi\xi}$, $y^{*}_{\xi\cxi} = y^{*}_{\cxi\xi}$ and $z^{*}_{\xi\cxi} = z^{*}_{\cxi\xi}$. In fact, from \eqref{CR-*} we have
\begin{equation} \label{mixed-derivatives}
\begin{array}{rcl}
y^*_{\xi\cxi} &=& i e^{\varphi} z_{\cxi}x_{\xi}+ie^{\varphi}x_{\xi\cxi} \\
y^*_{\cxi\xi} &=& -i e^{\varphi} z_{\xi}x_{\cxi}-ie^{\varphi}x_{\cxi\xi}
\end{array}
\end{equation}
Then 
$$y^*_{\xi\cxi}-y^*_{\cxi\xi} = i e^{\varphi} \left[\dot{\varphi}(z_{\cxi} x_{\xi} + z_{\xi} x_{\cxi} ) + 2x_{\cxi\xi} \right] = 0
$$
where we use the first equation of \eqref{eq:system-psi} in the last equality. Analogously, 
$$x^*_{\xi\cxi}-x^*_{\cxi\xi} = -i e^{\varphi} \left[\dot{\varphi}(z_{\cxi} y_{\xi} + z_{\xi} y_{\cxi} ) + 2y_{\cxi\xi} \right] = 0.$$ The equation $z^{*}_{\xi\cxi} = z^{*}_{\cxi\xi}$ is straightforward.
\end{proof}
As $\Sigma$ is simply-connected, from the above Proposition, we conclude that there exists $\Psi^* = (x^*, y^*, z^*): \Sigma\rightarrow \mathbb{R}^3$ such that \eqref{CR-*} holds for any complex parameter $\xi=u+ i v$.

In what follows let us consider $\L^3$  the Minkowski space $\R^3$ with the Lorentz metric 
\begin{equation} \ll \cdot ,  \cdot \gg = dx^2 + dy^2 - dz^2,\label{mmetric}\end{equation}and $\|V\|:=\sqrt{|\ll V, V \gg|}$ the norm of vector $V$ in $\mathbb{L}^3$.

\begin{proposition}\label{confpara}
Let $\psi=(x,y,z) : \Sigma\rightarrow \mathbb{R}^3$ be a  $\phie3$-minimal immersion, $\xi=u+iv$ be a conformal parameter for $\psi$  and   $\psi^*=(x^*,y^*,z^*) : \Sigma \rightarrow \mathbb{L}^3$ be the map given by the system \eqref{CR-*}. Then, $\xi$ is also a conformal parameter for $\psi^*$,
$$ x_\xi^{*2} + y_\xi^{*2}-z_\xi^{*2}=0,$$
 and  
\begin{align} \label{eq:relation-EE*}
E^*&= 2(|x_\xi^{*}|^2 + |y_\xi^{*}|^2-|z_\xi^{*} |^2)=2e^{2\varphi} \eta^2(|x_\xi|^2 + |y_\xi|^2+|z_\xi|^2)\\
&=e^{2\varphi} \eta^2 E, \nonumber
\end{align} 
where $\eta=\langle N,\vec{e}_3\rangle$ is the angle function of $\psi$ with respect to the unit normal $N = - i \dfrac{\psi_\xi \wedge \psi_{\cxi}}{|\psi_\xi \wedge \psi_{\cxi} |}$ of $\psi$.

Furthermore, $\eta \neq 0$ ( i.e. $\psi=(x,y,z(x,y))$ is locally a vertical graph) if and only if   $\psi^*$ is an immersion.  In this case, a timelike  unit normal for $\psi^*$ is given by
\begin{equation}\label{normal*}
N^* = i \dfrac{\psi^*_\xi \wedge_{\mathbb{L}^3} \psi^*_{\cxi}}{||\psi^*_\xi \wedge_{\mathbb{L}^3} \psi^*_{\cxi}||} = \dfrac{1}{\eta E} (\eta_1, \eta_2, E) = (-z_x,-z_y,\sqrt{1 + z_x^2 + z_y^2}\,),
\end{equation} 
where $\eta>0$, $- 2 i \psi_\xi \wedge\psi_{\cxi}=  (\eta_1,\eta_2,\eta_3)$ and  $ \wedge_{\mathbb{L}^3} $ is the standard cross product in $\mathbb{L}^3$ given by $$\ll\, A \wedge_{\mathbb{L}^3} B,C\,\gg=-det(A,B,C).$$\end{proposition}
\begin{proof}
From the system \eqref{CR-*} we have,
\begin{equation} \label{coordinates-psi-*}
\psi_\xi^* = e^{\varphi} (-i y_\xi, i x_\xi, z_\xi), \end{equation}
Since  $\xi = u+ i v$ is a conformal parameter of $\psi$, we have from \eqref{coordinates-psi-*}, that
$$
 x_\xi^{*2}+y_\xi^{*2}-z_\xi^{*2} = - e^{2\varphi}(x_\xi^2+y_\xi^2+z_\xi^2) = 0.
$$
Therefore, $\xi = u+iv$ also provides a conformal parameter for $\psi^*$. 
Let us take a further view on $ E^*$. Firstly, if we write
$$
- 2 i\psi_\xi \wedge \psi_{\cxi} = - 2 i  \left( y_\xi z_{\cxi} - y_{\cxi} z_\xi,x_{\cxi} z_\xi - x_\xi z_{\cxi}, x_\xi y_{\cxi} - x_{\cxi}y_\xi \right) =  \left( \eta_1, \eta_2, \eta_3 \right),
$$
 a unit normal for $\psi$ is given by
\begin{equation} \label{normal-psi}
N = \dfrac{1}{E} \left( \eta_1, \eta_2, \eta_3 \right).
\end{equation}
But, from \eqref{CR-*} and  since $ x_\xi^2 + y_\xi^2 + z_\xi^2=0$, we have that 
$$
\begin{array}{rcl}
 \eta_3^2 &=&8 |x_\xi|^2 |y_\xi|^2 - 4x_\xi^2y_{\cxi}^2 -4 x_{\cxi}^2 y_\xi^2 \\
&=&  4 (|x_{\xi}|^2 + |y_{\xi}|^2 - |z_\xi|^2)(|x_{\xi}|^2 + |y_{\xi}|^2 + |z_\xi|^2)\\
&=&  e^{- 2\varphi} E E^*.
\end{array}
$$
It follows from \eqref{normal-psi} that  $\eta = \dfrac{\eta_3}{E}$ and we have the relation \eqref{eq:relation-EE*} which also says  that $\psi^*$ is an immersion if and only if $\eta\neq 0$.

\

Now we proceed in determining a unit normal for $\psi^*$ when $\eta\neq 0$. Consider  $V=(\eta_1,\eta_2,E)$. Since $ E=2 (|x_\xi|^2+|y_\xi|^2+|z_\xi|^2)$ and $x_\xi^2 + y_\xi^2 + z_\xi^2=0$,  from \eqref{CR-*} and by a straightforward computation, we have
$$
\ll \psi_\xi^*,V\gg =\ll \psi_{\cxi}^*,V\gg  = 0.
$$
But $E^2 = \eta_1^2+\eta_2^2+ \eta_3^2$ and then,
\begin{equation}
\ll V, V \gg = \eta_2^2+\eta_1^2-E^2 = -\eta_3^2.
\end{equation}
Consequently, for $\eta_3 \neq 0$, we may assume that  a unit normal for $\psi^*$ is given by
\begin{equation} \label{normal-*}
N^* = \dfrac{1}{\eta_3} (\eta_1, \eta_2, E), \quad \text{with $\eta_3>0$}.
\end{equation}
Furthermore, if we write $2 i \psi_{\xi}^* \wedge_{\mathbb{L}^3} \psi_{\cxi}^* = (\eta_1^*, \eta_2^*, \eta_3^*)$ then 
$$
\begin{array}{rcl}
\eta_3^*= -2 i \ll \vec{e}_3, \psi_\xi^* \wedge_{\mathbb{L}^3} \psi_{\cxi}^* \gg &=& - 2 i e^{2\varphi} (x_\xi y_{\cxi} - y_\xi x_{\cxi})= e^{2\varphi} \eta_3,
\end{array}
$$
and  since $2 ||\psi_\xi^* \wedge_{\mathbb{L}^3} \psi_{\cxi}^*||=E^{*}$, we have
$$
N^*= \frac{1}{\eta E}(\eta_2,\eta_1,E) =  \dfrac{\psi_v^* \wedge_{\mathbb{L}^3} \psi_u^*}{||\psi_u^* \wedge_{\mathbb{L}^3} \psi_v^*||} .
$$
The last equality in \eqref{normal*} is straightforward.
\end{proof}

Assume now that $\psi$ is a $\phie3$-minimal vertical graph $z=z(x,y)$ over some planar domain $\Omega\subseteq \mathbb{R}^2$ and $\xi=u+iv$ a conformal parameter. Then, by the uniformization theorem for nonanalytic metrics, \cite{S}, we know that the change of coordinates $$ \Phi: \Omega \rightarrow {\cal G}=\Phi(\Omega), \qquad (x,y)\rightarrow \Phi(x,y)=(u(x,y),v(x,y)),$$
is a ${\cal C}^2$-diffeomorphism with positive Jacobian and the following  Beltrami system is satisfied
\begin{align}\label{beltrami}
& y_\xi = \frac{1}{\alpha}(\beta - i W)x_\xi,
\end{align}
where $\alpha =1+z_y^2$,  $\beta=-z_x z_y$,  $W=\sqrt{\alpha \varrho - \beta^2}$ and $\varrho= 1 + z_x^2$.

From \eqref{beltrami}, we have
\begin{equation}\label{relaciones}
\begin{array}{rcl}
|y_\xi|^2 & = & \displaystyle \frac{\varrho}{\alpha} \ |x_\xi|^2\vspace{.05in}\\
 y_\xi x_{\cxi}  + y_{\cxi} x_\xi & = & \displaystyle\frac{2\beta}{\alpha} \ |x_\xi|^2\vspace{.05in}\\
y_\xi x_{\cxi}  - y_{\cxi} x_\xi  & = &-i \displaystyle\frac{2W}{\alpha}   \ |x_\xi|^2,
\end{array}
\end{equation}
and  since 
$$
\xi_x =\frac{y_{\cxi}}{ y_{\cxi} x_\xi -y_\xi x_{\cxi}  }, \qquad \xi_y =-\frac{x_{\cxi}}{ y_{\cxi} x_\xi -y_\xi x_{\cxi}  },
$$
then, from \eqref{CR-*}, \eqref{beltrami} and \eqref{relaciones} we obtain the following relations are satisfied,
\begin{equation}\label{partialsxy}
\begin{array}{ll}
x^*_x =-\frac{\varrho}{W}e^\varphi, & x^*_y =\frac{\beta}{W}e^\varphi\vspace{.05in}\\
y^*_x =\frac{\beta}{W}e^\varphi, &y^*_y =-\frac{\alpha}{W}e^\varphi\vspace{.05in}\\
 z^*_x= e^\varphi z_x, &z^*_y=e^\varphi z_y.
 \end{array}
\end{equation}
Now, from \eqref{partialsxy} and by using the expression of the normal $N^*$ in \eqref{normal*} , we can prove that the second fundamental forms of $\psi$ and $\psi^*$ are related as follows,
\begin{equation}\label{sff}
\begin{array}{rcccl}
-\ll \psi^*_{xx},N^*\gg&=&-\ll \psi^*_{x},N^*_x\gg&=&e^\varphi\langle \psi_{xx},N\rangle,\\
-\ll \psi^*_{xy},N^*\gg&=&-\ll \psi^*_{x},N^*_y\gg&=&e^\varphi\langle \psi_{xy},N\rangle,\\
-\ll \psi^*_{yy},N^*\gg&=&-\ll \psi^*_{y},N^*_y\gg&=&e^\varphi\langle \psi_{yy},N\rangle.
\end{array}
\end{equation}
The equations \eqref{eq:relation-EE*} and \eqref{sff} allow us to easily prove,
\begin{proposition}
At regular points, the mean curvature $H^*$ and the Gaussian curvature $K^*$ of $\psi^*$ are given by
\begin{equation} \label{meancurvature-3rd-step}
e^{\varphi}\eta^2 H^* =  H,
\end{equation}
and
\begin{equation} \label{gausscurvature-final-step}
e^{\varphi}\eta^4 K^* =  -K.
\end{equation}
where $H$ and $K$ denote the mean curvature and the Gaussian curvature of $\psi$, respectively.
\end{proposition}
Furthermore, equation \eqref{meancurvature-3rd-step} tell us that, if $\psi$ is $\phie3$-minimal, we have at the regular points of $\psi^*$ that
\begin{equation}\label{psi*-maximal-1}
H^* = \dfrac{\dot{\varphi}e^{-\varphi}}{\eta} =  \dot{\varphi}e^{-\varphi} \eta^*,
\end{equation}
where $\eta^* =- \ll N^*, \vec{e}_3 \gg = \dfrac{1}{\eta}$. Now, by using the third equation of \eqref{CR-*} and taking  $\varphi^*(z^*)=\varphi(z(z^*))$ we have,
\begin{equation} \label{psi*-maximal-2}
\dot{\varphi^*}=\dfrac{d\varphi^*}{dz^*} = \dfrac{d\varphi}{dz}\dfrac{dz}{dz^*} = \dot{\varphi}e^{-\varphi}
\end{equation}
and, from \eqref{psi*-maximal-1} and \eqref{psi*-maximal-2}, we conclude that 
\begin{equation} \label{psi*-maximal-3}
H^* = \dot{\varphi^*} \eta^*,
\end{equation}
and $\psi^*$ is a $[\varphi^*,\vec{e}_3]$-maximal surface or, more precisely, $\psi^*$ is a $[\varphi \circ (z^*)^{-1},\vec{e}_3]$-maximal surface. 
\begin{definition} The pair $(\psi,\psi^*):\Sigma\rightarrow \R^3 \times \mathbb{L}^3$ will be called a Calabi's pair of weight $\varphi$.
\end{definition}
\begin{remark} If $(\psi,\psi^*)$ is a Calabi's  pair of weight $\varphi$, then from \eqref{eq:system-psi}, \eqref{CR-*} and for any complex parameter $\xi=u+iv$, $\psi^*$ satisfies the following system:
\begin{equation}\label{system-psi*}
\begin{array}{rcl}
2 x^*_{\xi\cxi}-\dot{\varphi^*} (z^*_{\xi} x^*_{\cxi}+z^*_{\cxi}x^*_{\xi}) &=& 0, \vspace{.05in}\\
2 y^*_{\xi\cxi}-\dot{\varphi^*} (z^*_{\xi} y^*_{\cxi}+z^*_{\cxi}y^*_{\xi}) &=& 0,  \vspace{.05in}\\
2 z^*_{\xi\cxi} - \dot{\varphi^*} (|x^*_{\xi}|^2+|y^*_{\xi}|^2+|z^*_{\xi}|^2)&=&0,
\end{array}
\end{equation}
where $\varphi^*(z^*) = \varphi(z)$.
\end{remark}
\begin{definition}\label{maxfaces}
Let $\Sigma$ be a Riemann surface, a differentiable  map  $\psi^* :(x^*,y^*,z^*): \Sigma \rightarrow \mathbb{L}^3$ is called a spacelike $[\varphi^*,\vec{e}_3]$-maxface if for any complex parameter $\xi=u+iv$ in $\Sigma$,  $\psi^*$ is a solution of \eqref{system-psi*} and 
\begin{align}
&x_\xi^{*2} + y_\xi^{*2}-z_\xi^{*2}=0,\label{conformal}\\
&|x^*_{\xi}|^2+|y^*_{\xi}|^2-|z^*_{\xi}|^2 \not \equiv 0, \label{eq:singularities} \\
&|x^*_{\xi}|^2+|y^*_{\xi}|^2+|z^*_{\xi}|^2 \neq 0.\label{eq:branchpoints}
\end{align}
A point $p$ where $(|x^*_{\xi}|^2+|y^*_{\xi}|^2-|z^*_{\xi}|^2)(p) = 0$ is called a singular point and $\psi^*(p)$ is a singularity of $\psi^*$. Otherwise, $p$ is a regular point.
\end{definition}

\begin{remark} \label{r-2}Note that, taking on $\Sigma$ the conformal structure determined by the induced metric, any spacelike $[\varphi^*,\vec{e}_3]$-maximal surface  $\psi^*:\Sigma\rightarrow\mathbb{L}^3$ is a $[\varphi^*,\vec{e}_3]$-maxface . \end{remark}
\begin{remark}\label{r-*2}
Condition \eqref{eq:singularities} allows a spacelike $[\varphi^*,\vec{e}_3]$-maxface to have singularities. Condition \eqref{eq:branchpoints} indicates that, a point where $|x^*_{\xi}|^2+|y^*_{\xi}|^2-|z^*_{\xi}|^2 = 0$ is not a branch point.
\end{remark}
\begin{remark} \label{r-3}Observe that if   $\Sigma$ is simply-connected and  $\psi^* :(x^*,y^*,z^*): \Sigma \rightarrow \mathbb{L}^3$ is a spacelike  $[\varphi^*,\vec{e}_3]$-maxface, then there exists $\psi=(x,y,z):\Sigma\rightarrow\mathbb{R}^3$  such that for any complex parameter $\xi= u+ iv$, $\psi$ satisfies the system 
\begin{equation} \label{CR-psi}
\begin{array}{rcl}
x_{\xi} &=&  i e^{-\varphi^*} y^{*}_{\xi}, \vspace{.05in}\\
 y_{\xi} &=&- i e^{-\varphi^*} x^{*}_{\xi}, \vspace{.05in}\\
z_{\xi} &=& e^{-\varphi^*}z^{*}_{\xi} .
\end{array}
\end{equation}and, consequently,  $(\psi,\psi^*)$ is a Calabi's pair of weight $\varphi$ (with $\varphi(z)=\varphi^*(z^*)$).
\end{remark}
\begin{remark}\label{r-4}By using \eqref{CR-*}, one may check that if $(\psi,\psi^*):\Sigma \rightarrow \R^3\times \mathbb{L}^3$ is a Calabi's pair and $N$ is the unit normal vector field to $\psi$ given by \eqref{normal-psi}, then the following expression holds
\begin{equation}\label{psi-psi*}
d\psi^* = \langle \vec{e}_3,d\psi\rangle (N + \vec{e}_3) - \langle \vec{e}_3,N\rangle d\psi.
\end{equation}
\end{remark}
\section{Weierstrass-type representation}\label{sec-3}
In this section we are going to show that when $\dot{\varphi}\neq 0$,  $\phie3$-minimal surfaces and $[\varphi^*,\vec{e}_3]$-maxfaces can be represented in terms of its Gauss map.
\subsection{The case of $\phie3$-minimal surfaces}
In \cite{kenmotsu}, Kenmotsu studied conformally parametrized surfaces in $\R^3$ in terms of its Gauss map. To be more precise, if $\psi=(x,y,z):\Sigma\rightarrow \R^3$ is an immersion with a unit normal vector field $N=(N_1,N_2,N_3)$  and Gauss map $g: \Sigma\rightarrow \overline{\mathbb{C}}$,
$$ g = \frac{N_1 - i N_2}{1 + N_3},$$
$N$ is written as 
$$ N= \left( 2 \frac{{\rm Re\,}{g}}{1 + |g|^2}, - 2 \frac{{\rm Im\,}g}{1+|g|^2}, \frac{1-|g|^2}{1+|g|^2}\right),$$
and  $\Sigma$ can be regarded as a Riemann surface with the conformal structure determined by the induced metric $ds^2$. 
Let $\xi=u+iv$ be a complex parameter of $\Sigma$, it was proved in \cite[Theorem 2, Theorem 3]{kenmotsu} that 
\begin{equation} \label{weierstrass-ken}
\begin{array}{rcl}
H x_{\xi} &=&- \dfrac{2\   (1-g^2)}{(1+|g|^2)^2} \bar{g}_{\xi}, \vspace{.05in}\\
H y_{\xi} &=& \dfrac{2 i \ (1+g^2)}{(1+|g|^2)^2} \bar{g}_{\xi},\vspace{.05in}\\
H z_{\xi} &=& \dfrac{4\  g}{(1+|g|^2)^2} \bar{g}_{\xi} , \\
\end{array}
\end{equation}
where $H$ is the mean curvature of $\psi$. Moreover, the complete integrability condition for the system \eqref{weierstrass-ken} is the following PDE
\begin{equation}\label{int-con}
H \left(g_{\xi\cxi} - \frac{2 \bar{g}}{1 + g\bar{g}} g_\xi g_{\cxi}\right)= H_\xi \, g_{\cxi}.
\end{equation}
In particular, if $\Sigma$ is simply-connected and $\psi=(x,y,z):\Sigma\rightarrow\R^3$ is a $\phie3$-minimal immersion with $\dot{\varphi}\neq 0$
we may assume (locally) that $\varphi$ is a solution of the following ODE,
\begin{equation}\label{phi-eq}
\dot{\varphi} = \epsilon e^{-\phi(\varphi)}, \quad \epsilon\in \{-1,1\},
\end{equation}
for some differentiable real function $\phi$. 
Now, from \eqref{weierstrass-ken}, \eqref{int-con}, by using that $$H=\dot{\varphi}N_3 =  \epsilon e^{-\phi(\varphi\circ z)}\, \frac{1-|g|^2}{1+|g|^2},$$ and by a straightforward computation, 
we have 
\begin{theorem}\label{weierstrass}
The $1$-form $\partial \bar{g} = \bar{g}_\xi d\xi $ vanishes at a point of $\Sigma$ if and only if $1=|g|^2$ at this point. If we take $\omega$ as,
\begin{equation}\label{omega}\omega = \frac{2\,\partial  \bar{g}}{1-|g|^4} =  \frac{2\, \bar{g}_\xi}{1-|g|^4} d\xi,
\end{equation}
then $\omega$, $g \omega$, $g^2 \omega$ are well-defined on $\Sigma$ and $(1+|g|^2)\omega\neq0$  everywhere. Furthermore, the immersion  $\psi$ can be recovered in terms of $\phi$  and  its Gauss map   as follows
\begin{equation} \label{weierstrass-phiminimas}
\begin{array}{rclcl}
 x_{\xi} d\xi&=&\dfrac{-1}{\dot{\varphi}}(1-g^2) \omega&=&-2 \epsilon e^{\phi(\varphi\circ z)}\,\dfrac{1-g^2}{1-|g|^4} \bar{g}_{\xi} d\xi, \vspace{.05in}\\
y_{\xi}d\xi &=&\dfrac{i}{\dot{\varphi}}(1+g^2)\omega&=&  2 i  \epsilon e^{\phi(\varphi\circ z)}\dfrac{1+g^2}{1-|g|^4} \bar{g}_{\xi}d\xi,\vspace{.05in}\\
z_{\xi} d\xi&=&\dfrac{2}{\dot{\varphi}}  g\, \omega&=&4 \epsilon e^{\phi(\varphi\circ z)}\dfrac{g}{1-|g|^4} \bar{g}_{\xi}d\xi. 
\end{array}
\end{equation}
where, $\varphi\circ z$  is also determined by $g$ as
\begin{equation}\label{varphi}\varphi \circ z= 4 \, {\rm Re} \int g\, \omega.\end{equation}
 The complete integrability condition for the system \eqref{weierstrass-phiminimas} is the following PDE
\begin{equation}\label{int-phiminimas}
g_{\xi\cxi} +\frac{2 |g|^2}{1- |g|^4}  \bar{g} \, g_\xi \, g_{\cxi}+ 2(2\phi^\prime(\varphi\circ z)+1)\,\frac{g} {1- |g|^4}\, |g_{\cxi}|^2 = 0,
\end{equation}
where $\phi^\prime$ denotes the derivative function of $\phi$. From \eqref{weierstrass-phiminimas}, the induced metric $ds^2$ is given by
\begin{equation}\label{metric}
ds^2 =  4 \, e^{2\phi(\varphi\circ z)}  (1+ |g|^2)^2 |\omega|^2 = 16\, e^{2\phi(\varphi\circ z)}\frac{ |\bar{g}_\xi|^2 |d\xi|^2}{(1-|g|^2)^2}
\end{equation}
Conversely, any $\phie3$-minimal surface satisfying \eqref{phi-eq} can be locally represented in this way. 
\end{theorem}
\begin{remark}\label{r-4}{\rm 
We remark  that  translating solitons in $\R^3$ are obtained when  $\phi^\prime\equiv 0$ and singular minimal surfaces when $\phi^\prime$ is a non-zero constant.  In these two cases a  Weierstrass-type representation was also studied in \cite{MMT-JMAA}.}
\end{remark}
\subsection{The case of spacelike $[\varphi^*,\vec{e}_3]$-maxfaces}
Let $\Sigma$ be a simply-connected Riemann surface and $\psi^*= (x^*,y^*,z^*):\Sigma \rightarrow \L^3$ a spacelike $[\varphi^*,\vec{e}_3]$-maxface where $\varphi^*$ is a solution of the following EDO,
\begin{equation} \label{punto-*}
\dot{\varphi}^* = \epsilon e^{-\phi^*(\varphi^*)}, \quad \epsilon\in\{-1,1\},\end{equation} for some differentiable  real function $\phi^*$. Then, from  Remark \ref{r-3}, there exists  $\psi=(x,y,z):\Sigma\rightarrow\R^3$ a $\phie3$-minimal immersion 
with
\begin{align}\label{relations}& dz^* = e^\varphi dz, \quad \varphi^*\circ z^* =\varphi \circ z,\\
&  \dot{\varphi} = \epsilon e^{-\phi(\varphi)}, 
\quad  \phi^*(\alpha) = \phi(\alpha) + \alpha
\end{align}
and such that   $(\psi,\psi^*)$ is a Calabi's pair satisfying  \eqref{CR-psi}. Consequently, from Theorem \ref{weierstrass} we also have the following Weierstrass-type representation for $\psi^*$.
\begin{theorem}\label{weierstrass-*}  $\psi^*$  can be  recovered in terms of  $\phi^*$  and the Gauss map  $g$ of $\psi$ as follows
\begin{equation} \label{weierstrass-phi*}
\begin{array}{rcl}
 x^*_{\xi} &=&2\epsilon \, e^{\phi^*(\varphi^*\circ z^*)}\dfrac{1+g^2}{1-|g|^4} \bar{g}_{\xi}, \vspace{.05in}\\
y^*_{\xi}&=& -2\, i \, \epsilon  e^{\phi^*(\varphi^*\circ z^*)} \dfrac{1-g^2}{1-|g|^4} \bar{g}_{\xi},\vspace{.05in}\\
z^*_{\xi} &=& 4  \epsilon \, e^{\phi^*(\varphi^*\circ z^*)}\dfrac{ g}{1-|g|^4} \bar{g}_{\xi} , \\
\end{array}
\end{equation}
where 
\begin{equation}\label{varphi-*}
 \phi^*\circ z^* =4 {\rm Re}\,\int g \omega,
\end{equation}  $\omega$ is as in  \eqref{omega} and  $g$ is a solution of 
\begin{equation} \label{int-phi*}
g_{\xi\cxi} +\frac{2 |g|^2}{1- |g|^4}  \bar{g} \, g_\xi \, g_{\cxi}+ 2(2\dot{\phi}^{*}(\varphi^*)-1)\,\frac{g} {1- |g|^4}\, |g_{\cxi}|^2 = 0,
\end{equation}
Conversely, if $\varphi^*$ is a solution of \eqref{punto-*} and   $(x^*,y^*,z^*)$ satisfies  \eqref{weierstrass-phi*}, then 
$ \phi^*\circ z^*$ is given by \eqref{varphi-*}
and   $g$ is a solution of \eqref{int-phi*}. Furthermore,  if $(1+ |g|^2)\omega\neq0$ everywhere,  then $\psi^*=(x^*,y^*,z^*)$ is  a $[\varphi^*,\vec{e}_3]$-maxface in $\L^3$.
\end{theorem} 
\begin{remark}\label{r-5}{\rm 
As in  remark \ref{r-4},  translating solitons and singular maxfaces  in $\L^3$ are obtained when  $\phi^{*\prime}$ is constant.}
\end{remark}
\begin{remark}\label{r-6}{\rm 
From  \eqref{weierstrass-phi*}, 
$$ |x^*_\xi|^2 + |y^*_\xi|^2 - |z^*_\xi|^2 = 8 \,e^{2\phi^*(\varphi^*\circ z^*)} \dfrac{|\bar{g}_\xi|^2}{(1 + |g|^2)^2}  $$  and then, the singular set ${\cal S}$ of $\psi^*$ is the set of points where  $|g|^2=1$.

At the regular points, if $N^*:\Sigma\setminus {\cal S} \rightarrow \mathbb{H}^2=\{(a,b,c)\in \L^3 \ : \ a^2 + b^2 - c^2 = -1\}$ denotes the upwards-pointing unit normal vector field of  $\psi^*$,  the Gauss map of $\psi^*$
is the composition $\pi \circ N^*$ where $\pi$ is the usual stereographic projection in $\L^3$ given by 
$$ \pi(a,b,c)= \frac{a - i b}{1+c}.$$ Thus, from  \eqref{normal-*}, we have that $g$ is also the Gauss map of $\psi^*$ at regular points and it is  well-defined on the singular set.}
\end{remark}
\section{Singularities of $[\varphi^*,\vec{e}_3]$-maxfaces}\label{sec-4}

In the sequence, following \cite{UY-HM, FSUY}, we provide a brief introduction to frontals and some criteria for admissible singularities.

Let $U \subset \mathbb{R}^2$ be an open subset and $f : U \rightarrow \mathbb{R}^3$ a smooth map. We are going to identify the  the unit cotangent bundle $T^*_1 \R^3$ with  $\R^3\times \mathbb{S}^2=\{({\bf x},\nu)\ : \ {\bf x}\in \R^3, \nu\in \mathbb{S}^2\} $. Then $\langle d{\bf x}, \nu\rangle$ is a contact form and the map $\chi = (f,\nu) : U\subseteq \R^2 \rightarrow \R^3\times \mathbb{S}^2$ is called Legendrian when the pull-back of $\langle d{\bf x}, \nu\rangle$ vanishes, that is, if $\nu$ is orthogonal to $df(TU)$.  In this case, its projection $f$ is called a frontal with unit normal $\nu$. When $\chi$ is a Legendrian immersion  $f$ is called a front. The set of points ${\cal S}\subseteq U$ where $f$ is not an immersion is called the singular set of $f$, and we can write
$$ {\cal S}=\{ p\in U \ : \ \lambda(p) := \langle f_u\times f_v,\nu\rangle(p) = 0\}, \quad (u,v)\in U.$$ By means of the function $\lambda$ defined above, a singular point $p \in U$ is called non-degenerate when $d\lambda$ is not identically null at $p$. 

Let $p$ be a non-degenerate singular point. Then the singular set ${\cal S}=\left\{ \lambda = 0 \right\}$ is given by a regular curve in a neighborhood of $p$. Such a curve is called singular curve. Let $\alpha(t) : I_\varepsilon \rightarrow U$ a local parametrization for it, i.e., $\alpha(0)=p$ and $\lambda(\alpha(t))=0, \, t \in I_\varepsilon,$  where by $I_\epsilon$ we denote the interval $=(-\varepsilon, \varepsilon)$.
The tangential vector $\alpha'(t)$ of $\alpha$ at $t$ is a singular direction. The condition $d\lambda_p \not \equiv 0$ implies that $f_u$ and $f_v$ do not vanish at $p$. Therefore the kernel of $df_p$ has dimension 1 when $p$ a non-degenerate singular point. In this case, a non-zero tangent vector $\varrho \in T_pU \in \textnormal{ker}(df_p)$ will be called a null direction. Let $\varrho(t)$ be a smooth vector field $\varrho(t)$ along $\alpha(t)$ that assigns for each $t$ a null direction at $\alpha(t)$ (which is unique up to scalar multiplication). We call it a vector field of null directions. 

In  \cite{UY-HM, FSUY} was proved the  following criteria for some admissible singularities of  frontals:
\begin{lemma}[\cite{FSUY}]\label{cuspidal-swallowtail}
Let $f : U \rightarrow \R^3$ be a front and $p \in U$ a non-degenerate singular point. Take a singular curve $\alpha(t)$ with $\alpha(0) =p$ and a vector field of null directions $\varrho(t)$ along a $\alpha(t)$. Then
\begin{enumerate}
\item The germ of $f$ at $p=\alpha(0)$ is a cuspidal edge if and only if the null direction $\varrho(0)$ is transversal to the singular direction $\alpha'(0)$.
\item The germ of $f$ at $p = \alpha(0)$ is a swallowtail if and only if the null direction $\varrho(0)$ is proportional to the singular direction $\alpha'(0)$ and
$$
\left. \dfrac{d}{dt} \right|_{t=0} \det (\varrho(t),\alpha'(t)) \neq 0,
$$
\end{enumerate}
where $\varrho(t)$ and $\alpha'(t) \in T_{\alpha(t)}U$ are considered as column vectors, and {\rm det} denotes the determinant of a $2 \times 2$-matrix.
\end{lemma}

\begin{lemma}[\cite{UY-HM}]\label{lemma-crosscap}
Let $f : U \rightarrow \R^3$ be a frontal with unit normal vector field $\nu$  and $\alpha(t)$ be a singular curve on $U$ passing through a non-degenerate singular point $p=\alpha(0)$. We set
$$
\delta(t):=\langle\tilde{\alpha}'\times D^f_{\eta} \nu, \nu\rangle,
$$
where $\tilde{\alpha} = f \circ \alpha$, $D^f_{\varrho} \nu$ is the canonical covariant derivative along  $f$ in $\R^3$, and $' = \frac{d}{dt}$. Then the germ of $f$ at $p=\alpha(0)$ is a cuspidal cross cap if and only if
\begin{enumerate}
\item $\varrho(0)$ is transversal to $\alpha'(0)$.
\item $\delta(0)=0$ and $\delta'(0) \neq 0$.
\end{enumerate}
\end{lemma}

\subsection{Criteria for singularities of $[\varphi^*,\vec{e}_3]$-maxfaces}
We will consider  $\Sigma$  a simply-connected Riemann surface and  $(\psi,\psi^*) : \Sigma \rightarrow \R^3\times\mathbb{L}^3$ a Calabi's pair of a real analytic weight $\varphi$. Then, as $\psi$ and $\psi^*$ are, respectively,  solutions of \eqref{eq:system-psi} and \eqref{system-psi*}, we know that $\psi$ and $\psi^*$ are also analytic, that is, $\psi,\psi^*\in {\cal C}^\omega(\Sigma)$ and, in particular, $\psi$ has a well-defined  analytic unit vector field $N\in  {\cal C}^\omega(\Sigma)$.  
\begin{theorem}\label{theorem-frontal}
 The $ {\cal C}^\omega$-map $\psi^*:\Sigma \rightarrow \mathbb{L}^3\simeq \R^3$ is a frontal in $\R^3$ with unit normal  $\nu$,  given by \begin{equation}\label{nu-*} 
\nu = \dfrac{N -(1+\eta)\vec{e}_3}{\sqrt{2-\eta^2}},
\end{equation}
where $  \eta =\langle N,\vec{e}_3,\rangle$ is the angle function of $\psi$.

The singular set ${\cal S}$ of $\psi^*$ is the set of points  $\{\eta=0\}$ where the angle function  vanishes. If $\alpha:I_\varepsilon\rightarrow \Sigma$ is a singular curve with $\alpha(0)=p$, which always exists by the analyticity of $\eta$, then $p$ is a non-degenerate singular point if and only if $\Gamma'(0)\neq0$, where $\Gamma = N\circ\alpha$.
\end{theorem}

\begin{proof}

Firstly, we check that $\psi^*$ is a frontal. In fact, let $\xi=u+ iv$ be a complex parameter of $\Sigma$, then from \eqref{eq:system-psi},   \eqref{CR-*} and \eqref{normal-psi}, 
\begin{align*} - 2 i \psi^*_\xi \wedge \psi^*_{\cxi} &= e^{2\varphi} ( 2 (x_\xi z_{\cxi} + x_{\cxi} z_\xi), 2 (y_\xi z_{\cxi} + y_{\cxi} z_\xi), - 2i(  x_\xi y_{\cxi} - x_{\cxi}y_\xi ))\\
& =e^{2\varphi}(-\dfrac{4}{\dot{\varphi}} x_{\xi\cxi}, -\dfrac{4}{\dot{\varphi}}y_{\xi\cxi}, \ \eta_3)=-e^{2 \varphi}E(\dfrac{\Delta_\psi x}{\dot{\varphi}},\dfrac{\Delta_\psi y}{\dot{\varphi}}, -\eta), \nonumber\end{align*} 
where $\Delta_\psi$ denotes the Laplace operator associated to the induced metric of $\psi$. 
Hence, since  $\psi$ is a $\phie3$-minimal immersion,  $\Delta_\psi \psi = H N= \dot{\varphi} \ \eta N$ and we have 
\begin{equation}\label{frontal}
- 2 i \psi^*_{\xi}\wedge \psi^*_{\cxi} = -e^{2\varphi} E \eta (N_1,N_2,-1),\end{equation}
 that is, $\psi^*: \Sigma \rightarrow \mathbb{L}^3\simeq \R^3$ is a frontal  with normal
 \begin{equation}\label{normal-frontal}
 \nu = \dfrac{1}{\sqrt{2-\eta^2}}(N - (1+\eta)\vec{e}_3).
 \end{equation}
 Thus, from \eqref{frontal},  $\{\eta = 0\}$ is the singular set of $\psi^*$ and then, a singular point $p$ is non-degenerate if and only if 
 $d\eta_p\neq 0$, that is,  if and only if $S\vec{e}_3\neq 0$, where by $S$ we denote the shape operator of $\psi$.
 
Since $H=0$ at the singular points, then, along $\gamma = \psi\circ \alpha$, we can write  that 
\begin{equation}\label{system-A}
\begin{array}{rcl}
S \vec{e}_3 &=&  \alpha_1 \,\vec{e}_3+ \alpha_2\,J \vec{e}_3, \\
S J\vec{e}_3 &=& \alpha_2 \,\vec{e}_3 - \alpha_1 \,J \vec{e}_3,
\end{array}
\end{equation}
where  $J\vec{e}_3 = \Gamma' \wedge \vec{e}_3$. Thus, by putting  $\gamma'= a \vec{e}_3+ b J\vec{e}_3$ and taking into account that $\langle \Gamma' ,\vec{e}_3\rangle =0$, we have
$$
- \Gamma' = S \gamma' = a S \vec{e}_3 + b SJ \vec{e}_3 = (a \alpha_1 + b \alpha_2) \vec{e}_3 + (a \alpha_2 - b \alpha_1) J \vec{e}_3 = \mu \|\gamma'\|^2 J\vec{e}_3,
$$
where $ \alpha_1 = -\mu b$ and $\alpha_2 =\mu \, a$ for some function $\mu$. We conclude that $p$ is degenerate (that is, $S\vec{e}_3=0$) if and only if $\Gamma'(0)=0$.
\end{proof}
\begin{theorem}[Criteria for singularities]\label{theorem-criteria}
Let  $p$ be a non-degenerate singular point of $\psi^*$ and   $\alpha: I_\varepsilon\rightarrow \Sigma$ be the unique singular curve around $p$, with $\alpha(0)=p$. If  $\gamma=\psi\circ \alpha$, then
\begin{enumerate}[a)]
\item $\gamma' (0)\wedge e_3  \neq 0$ if and only if $\psi^*$ is a front at $p$. In this case,
\begin{enumerate}[i)]
\item  $\langle \gamma'(0), e_3 \rangle \neq 0$ if and only if $p$ is a cuspidal edge.
\item $\langle \gamma'(0), e_3 \rangle  = 0$ and $\langle \gamma''(0), e_3 \rangle  \neq 0$ if and only if $p$ is a swallowtail edge.
\end{enumerate}
\item We have a cuspidal cross cap if and only if $\gamma'(0) \wedge e_3 = 0$ and $\gamma''(0) \wedge e_3 \neq 0$.
\end{enumerate}
\end{theorem}
\begin{proof}
Since  $H=0$ along the singular set,   $\Delta_\psi \psi (\alpha(t)) = 0$, and then,  from \eqref{eq:system-psi} 
$$
d\psi^*(-z_v \partial_u + z_u \partial_v) = (-z_v x_v - z_u x_u, z_u y_u + z_v y_v, 0) = 0.
$$
In particular, we have that $$\varrho= -z_v \dfrac{\partial}{\partial u} + z_u \dfrac{\partial}{\partial v} $$ is a null direction along with $\alpha$. Furthermore, let us observe that 
$$d\psi(\varrho) = (-z_v x_u + z_u x_v, -z_v y_u + z_u y_v, 0) = | \psi_u \wedge \psi_v | N \wedge e_3$$ 
 and $d\psi (\varrho)$ can be also  considered as a null direction in $\R^3$  which is proportional to $ J\vec{e}_3=N \wedge e_3$.

Next,  we provide a necessary and sufficient condition for the frontal $\psi^*$ to be a front. If $X$ is a tangent vector such that $d\nu(X)=0$, then $\psi^*$ is a front at a neighbourhood of $p$ if and only if $X$ is not proportional to $N\wedge \vec{e}_3$. 

But, from \eqref{normal-frontal} we have that
 $d \nu (X) = 0$ if and only if $SX= \langle SX,\vec{e}_3\rangle \vec{e}_3$, that is, if and only if $S X \wedge e_3 = 0$. Writing $X = a_1\vec{e}_3+ a_2J\vec{e}_3$ we have from  \eqref{system-A} that
$$
S X = a_1 S \vec{e}_3 + a_2 SJ \vec{e}_3 = (a_1\alpha_1 + b \alpha_2) \vec{e}_3 + (a_1 \alpha_2 - a_2 \alpha_1) J \vec{e}_3.
$$

Thus, $SX \wedge \vec{e}_3 = 0$ if and only if $a_1 \alpha_2 - a_2 \alpha_1= 0$, which is equivalent to  $X \wedge S\vec{e}_3=0$, i.e., $X$ and  $S\vec{e}_3$ must be  proportional. Therefore, we have a front if and only if $S\vec{e}_3 \wedge (N \wedge\vec{e}_3) \neq 0$ or, equivalently, $\langle S \vec{e}_3, \vec{e}_3 \rangle \neq 0$.

Let us now show that $\langle S \vec{e}_3, \vec{e}_3 \rangle \neq 0$ if and only if $\gamma' \wedge \vec{e}_3 \neq 0$. On the one hand, let us suppose that $\gamma' \wedge \vec{e}_3 = 0$. Since $\langle N', \vec{e}_3 \rangle = 0$, we have
$
\langle S \gamma', \vec{e}_3 \rangle = 0$ and then $ \langle S \vec{e}_3, \vec{e}_3 \rangle = 0.
$
Therefore $\langle S \vec{e}_3, \vec{e}_3 \rangle \neq 0$ implies $\gamma' \wedge \vec{e}_3 \neq 0$.

On the other hand, consider that $\gamma' \wedge \vec{e}_3 \neq 0$ and suppose by contradiction that $\langle S \vec{e}_3, \vec{e}_3 \rangle = 0$. In this case, we may write $\gamma' = a_1 \vec{e}_3 + a_2 J \vec{e}_3$ with $a_2 \neq 0$. Since  $-N' = a_1 S\vec{e}_3 + a_2 SJ \vec{e}_3 $, we conclude that $0=-\langle N', \vec{e}_3 \rangle = a_2 \langle S J \vec{e}_3, \vec{e}_3 \rangle$,  and then $\langle S J \vec{e}_3, \vec{e}_3 \rangle = 0$.  It follows now that $S \vec{e}_3 = 0$ and the point $p$ is degenerated, which is a contradiction. Thus, $\langle S \vec{e}_3, \vec{e}_3 \rangle \neq 0$ and we conclude that $\psi^*$ is a front in a neighbourhood of $p$  if and only if $\gamma'(0) \wedge \vec{e}_3 \neq 0$.

As a next step we will construct the function $\delta(t)$ in Lemma \ref{lemma-crosscap}. Since $\langle \nu, \nu\rangle = 1$, we have $d\nu(J\vec{e}_3) = D^{\psi^*}_{J\vec{e}_3} \nu$ and  along the singular curve,
\begin{align}
&\nu= \dfrac{1}{\sqrt{2}} (N - \vec{e}_3),\label{nu-alpha}\\
&D^{\psi^*}_{J\vec{e}_3} \nu =  \dfrac{1}{\sqrt{2}} \langle SJ\vec{e}_3,J\vec{e}_3\rangle J\vec{e}_3 =\langle SJ\vec{e}_3,J\vec{e}_3\rangle\ \nu\wedge \vec{e}_3. \label{N*-derivative}
\end{align}

Using that $\langle N', \vec{e}_3\rangle =0$, \eqref{system-A}, \eqref{nu-alpha} and \eqref{N*-derivative} we obtain
$$
\begin{array}{rcl}
\delta(t) &=& \dfrac{1}{\sqrt{2}} \langle (\gamma^*)',\vec{e}_3\rangle  \langle SJ\vec{e}_3, J\vec{e}_3\rangle\vspace{.1in} \\
&=&\dfrac{e^\varphi}{\sqrt{2}} \langle (\gamma',\vec{e}_3\rangle  \langle SJ\vec{e}_3, J\vec{e}_3\rangle\vspace{.1in}\\
&=&\dfrac{e^\varphi}{\sqrt{2}} a_1\alpha_1= -\dfrac{e^\varphi}{\sqrt{2}}a_2 \alpha_2,
\end{array}
$$
where $\gamma^*=\psi^*\circ \alpha$.

We are now  in position to verify the conditions (i) and (ii) of Theorem \ref{theorem-criteria}. Regarding Lemma \ref{cuspidal-swallowtail}, if we write $\gamma'=a_1\vec{e}_3+ a_2J\vec{e}_3$, then  $\psi^*$ has a cuspidal edge in a neighbourhood of $p$ if and only if $a_1(0) \neq 0$, which is equivalent to $\langle \gamma'(0), \vec{e}_3\rangle \neq 0$. If $a_1(0)=0$, then we have a swallowtail edge at $p$  if and only if $a_1'(0)=0$. Therefore $p$ is  swallowtail edge singularity  if and only if $\langle \gamma'(0), \vec{e}_3\rangle = 0$ and $\langle \gamma''(0), \vec{e}_3 \rangle \neq 0$.   

Finally, for proving b) of Theorem \ref{theorem-criteria} we need to verify that the two conditions in Lemma \ref{lemma-crosscap} are equivalent to b).  In fact, for the first condition in Lemma \ref{lemma-crosscap},  $\gamma' \wedge J \vec{e}_3 \neq 0$ or equivalently $\langle \gamma', \vec{e}_3\rangle\neq 0$.  In this case, $\delta(0)=0$ if and only if $\alpha(0)=\langle S\vec{e}_3, \vec{e}_3 \rangle =0$, and we already know that this condition is equivalent to $\gamma'(0) \wedge \vec{e}_3 = 0$. To compute $\delta'(0)$ we use the expression $\delta(t) = -\dfrac{e^{\varphi}}{2} a_2 \alpha_2$. In this case, since $\delta(0)=0$, we obtain that $\delta'(0) = \dfrac{e^{\varphi}}{2} a_2'(0) \alpha_2(0)$. But $ \alpha_2(0)\neq 0$, otherwise $p$ is degenerate, thus, $\delta'(0) \neq 0$ if and only if $a_2'(0) \neq 0$. Finally, since $a_2(0)=0$ and $\gamma' = a_1 \vec{e}_3+ a_2 J\vec{e}_3$, we have that $\gamma''(0) = a_1'(0) \vec{e}_3 + a_2'(0) J\vec{e}_3$ and $\gamma''(0) \wedge \vec{e}_3 = - a_2'(0) N$. Consequently, $\delta'(0) \neq 0$ if and only if $\gamma''(0) \wedge \vec{e}_3 \neq 0$. We conclude that we have a cuspidal cross cap if $\gamma'(0) \wedge \vec{e}_3 = 0$ and $\gamma''(0) \wedge \vec{e}_3 \neq 0$.
\end{proof}
\subsection{The canonical conformal parametrization around a non-degenerate singular point}
Let $\alpha$ be a non-degenerate singular curve around a singular point $p\in \Sigma$. Since the angle function is analytic,  $\eta\in {\cal C}^\omega(\Sigma)$, we have that $\alpha$ must be   analytic and therefore there is a conformal parametrization $\zeta:=u+iv$,  $u\in I_\varepsilon:=(-\varepsilon,\varepsilon)$ around $p$ such that $\alpha$ can be parametrized as  $\alpha(u) = \zeta^{-1}(u,0)$, $u\in I_\varepsilon$. 

On the other side, from Theorem \ref{theorem-frontal}, we know that the curve  $\Gamma= N\circ \alpha$ is a regular curve  in $\R^3$ with a real analytic arc-length parameter $s(u)$. Thus,  by the real analyticity of $s(u)$,  there exists $\vartheta(\zeta) = s(\zeta)$, unique holomorphic function satisfying that  $\vartheta(u,0)=s(u)$, $u\in I_\varepsilon$. In other words, $\vartheta$ provides a new conformal parameters around p, which, due to abuse of notation, we will continue writing as $u+ i v$ in such a way that  for some small enough $\varepsilon>0 $    and $u\in I_\varepsilon$, we have that
\begin{align*}\Gamma(u) &= \dfrac{\psi_u \wedge \psi_v}{\|\psi_u \wedge \psi_v\|}(u,0)=N(u,0) = (\cos u, -\sin u,0),\\
\gamma'(u)&= \psi_u(u,0) = (\lambda_1(u),\lambda_2(u),A(u)),
\end{align*}
with $\lambda_1,\lambda_2,A\in {\cal C}^\omega(I_\varepsilon)$. Consequently,
$$ \psi_v(u,0) = \Gamma(u) \wedge \gamma'(u) = (-A(u)\sin u, -A(u) \cos u, B(u)), $$
with $B(u) = \lambda_2(u) \cos u + \lambda_1(u)\sin u$. Using now that $\langle \Gamma(u),\gamma'(u)\rangle=0$, we conclude that $\lambda_1(u) = B(u) \sin u$ and $\lambda_2(u) = B(u) \cos u$.

Taking into account \eqref{CR-*}  and as a consequence of the previous reasoning, we have the following proposition,\begin{proposition}\label{canonical-parametrization}If $\varphi\in{\cal C}^\omega(I)$ is real analytic and $(\psi,\psi^*):\Sigma \rightarrow \R^3\times \mathbb{L}^3$ is a Calabi's pair of weight   $\varphi$, then around any non-degenerate singular point $p$ of $\psi^*$ with $\langle\psi(p), \vec{e}_3\rangle\in I$, there is a unique conformal parametrization $$\zeta:{\cal U}(p)\subset \Sigma\rightarrow {\rm D}_\varepsilon = \{ (u, v) \ : \ \sqrt{u^2 + v^2} < \varepsilon\},\quad \zeta(p) = 0,$$ such that $\alpha(u) = \zeta^{-1}(u,0)$, $u\in I_\varepsilon$,  is the singular set of $\psi^*$ around $p$ and
\begin{align*}
\psi_u(u,0)&= (B(u) \sin u, B(u) \cos u, A(u)),\\
\psi_v(u,0)&= (-A(u) \sin u, -A(u) \cos u, B(u)),\\
\psi^*_u(u,0)&= A^*(u) (\cos u,  -\sin u, 1),\\
\psi^*_v(u,0) &= B^*(u)( \cos u, - \sin u, 1),
\end{align*}
where  $A,B\in{\cal C}^\omega(I_\varepsilon)$, $A^2 + B^2\neq 0$ and
\begin{align*}
& A^*= e^{\varphi\circ a} A, \quad B^*= e^{\varphi\circ a} B,\quad a(u) = \langle\psi(p), \vec{e}_3\rangle + \int_0^u A(t)dt, 
\end{align*}
\end{proposition}
\begin{definition} \label{def-canonical}The parametrization $\zeta$ in Proposition \ref{canonical-parametrization} will be called the canonical conformal parametrization around the non-degenerate singular point $p$.
\end{definition}
\begin{remark} {\rm Observe from \cite[Section 1.]{UY-fronts} that if a singular curve $\alpha(t)$  of $\psi^*$  consists only of cuspidal edges, then the function $A^*$ in Proposition \ref{canonical-parametrization} is, up to the sign,   $\dfrac{1}{2\kappa _s}$  where $\kappa_s$ is   the singular curvature of $\psi^*\circ\alpha$}.
\end{remark} 
\subsection{Classification of non-degenerate singularities}
In this Section  we will study, locally,  the space of  $[\varphi^*,\vec{e}_3]$-maxfaces with a prescribed  curve of non-degenerated singular points.  But first, we will show that $\phie3$-minimal surfaces are uniquely determined around a curve where the angle function vanishes identically and no tangent vector to the curve is an asymptotic direction.

As in the previous section, we will assume that $\varphi$ is  real analytic in some interval $I\subseteq \R$ and  $A,B:I_\varepsilon\rightarrow \R$ are  two real analytic functions with $A^2 + B^2>0$. Then, for any fixed point $(p_1,p_2,p_3)\in \R^3$, $p_3 \in I$,   the Cauchy problem
\begin{equation}\label{cauchy-psi}
\left\{
\begin{array}{l} 
\Delta \psi = -\dot{\varphi}(z) (\langle \nabla x,\nabla z\rangle,\langle \nabla y,\nabla z\rangle, - |\nabla x|^2- |\nabla y|^2+ |\nabla z|^2),\vspace{.1in}\\
\psi(u,0)  =  (p_1,p_2,p_3) +\displaystyle  \int_0^u(B(t) \sin t, B(t) \cos t, A(t)) dt\vspace{.1in}\\
\psi_v(u,0)  =   (-A(u) \sin u, -A(u) \cos u, B(u)),
\end{array}
\right.
\end{equation}
is well-posed on a small disk $D_R=\{(u,v) : \|(u,v)\|<R\}$ of radius $R$ with $R$ small enough,  
where $\Delta$ and $\nabla$ stand for  the usual Laplacian and Gradient operators in $D_R$.
 \begin{theorem}\label{existence-psi}  For  small enough $R$, there exists a unique real analytic solution 
 $\psi=(x,y,z): D_R\longrightarrow \R^3$ to \eqref{cauchy-psi} which defines a conformally immersed $[\varphi,\vec{e}_3]$-minimal surface in $\R^3$ whose angle function $\eta$ satisfies that $\eta(u,v)\neq 0$ if $v\neq0$, $\eta(u,0)=0$ and $\eta_v(u,0)=1$ for all $u$.
 \end{theorem}
 \begin{proof} Since $p_3 \in I$, we have existence and uniqueness of solution of \eqref{cauchy-psi} in $D_R$ for  small enough $R$.
 
  If $\xi= u+iv$, the conformal condition for $\psi$ is equivalent to the complex equation $\langle \psi_\xi,\psi_\xi\rangle =0$. But from the PDE in \eqref{cauchy-psi} we have 
 $$ \langle \psi_\xi,\psi_\xi\rangle _{\cxi}= 2\langle \psi_{\xi\cxi},\psi_\xi\rangle = -\dot{\varphi}(z) z_{\cxi} \langle \psi_\xi,\psi_\xi\rangle, $$
 that is, $ \langle \psi_\xi,\psi_\xi\rangle$ is  pseudo-holomorphic and its zeros must be isolated (see \cite[Theorem 6.1]{Bers}). But from \eqref{cauchy-psi}, $ \langle \psi_\xi,\psi_\xi\rangle(u,0) =0$ and we deduce that $ \langle \psi_\xi,\psi_\xi\rangle=0$ globally on $D_R$. Also taking into account  that  $\|\psi_u (u,0)\|^2 = A^2 + B^2>0$ and \eqref{eq:system-psi}, we   have that, for small enough $R$, $\psi:D_R \rightarrow \R^3$  is a conformal $[\varphi,\vec{e}_3]$-minimal immersion.

 On the other hand, if $$\eta=\langle N,\vec{e}_3\rangle=\langle \dfrac{\psi_u \wedge\psi_v}{\|\psi_u\|^2},\vec{e}_3\rangle$$ is the angle function then, from \eqref{cauchy-psi}, we have that $\eta(u,0)=0$ and 
 $$\eta_v(u,0) = \frac{1}{A^2 + B^2}\langle \psi_{uv}\wedge \psi_v + \psi_{uu}\wedge\psi_u,\vec{e}_3\rangle (u,0)= 1,$$
 which finishes the proof.
 \end{proof}
  \begin{figure}[h]
\begin{center}
\includegraphics[width=.30\textwidth]{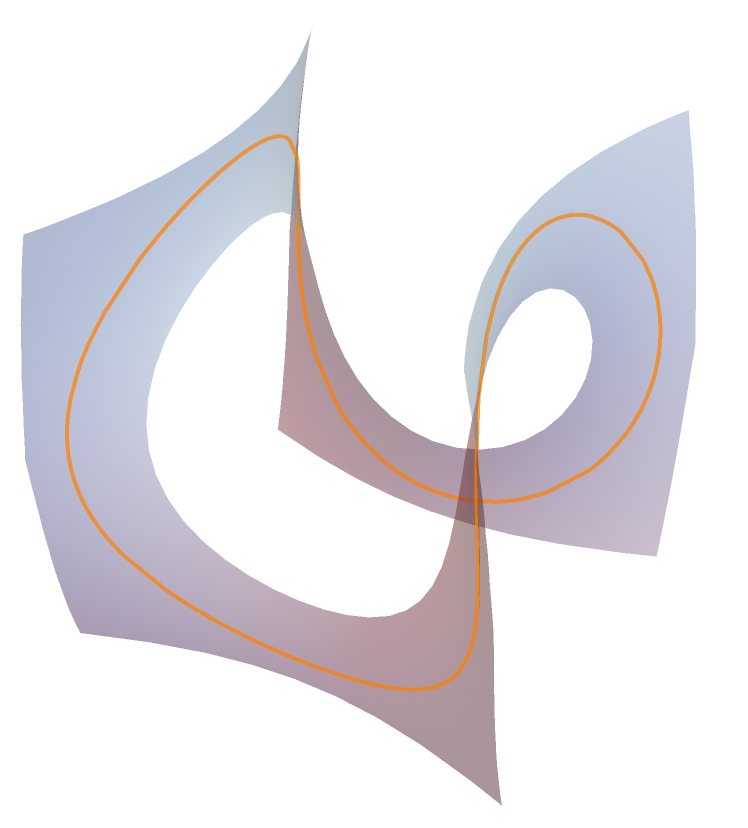} \hspace{1.5cm}
\includegraphics[width=.35\textwidth]{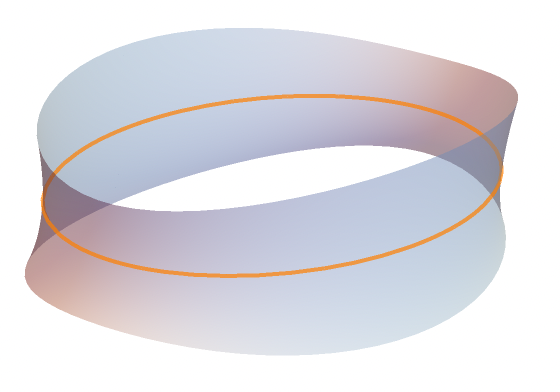}
\end{center}
\caption{ Solutions to \eqref{cauchy-psi} when $\dot{\varphi}\equiv1$   $(A(u)=\sin 2u, B(u) = \cos 2u)$ and $(A\equiv0, B(u) = 1)$ }
\end{figure}
 \begin{corollary}\label{eta-0} Let $\varphi$ be a real analytic function on an interval $I$. Then for any prescribed real analytic regular curve $\alpha: I_\varepsilon\rightarrow\R^3$ with  $\langle\alpha(0),\vec{e}_3\rangle\in I$, there exists a unique $\phie3$-minimal surface passing through $\alpha(t)$, $|t|<\varepsilon'$, with $\varepsilon'$ small enough,  such that, along $\alpha$,  $\eta =0$ and $d\eta \neq 0$.
 \end{corollary}
 \begin{proof}
 The existence follows from Theorem \ref{existence-psi} and the uniqueness comes from Propositions \ref{confpara}, \ref{canonical-parametrization} and the uniqueness in the Cauchy problem \eqref{cauchy-psi}.
 \end{proof}
 From Corollary \ref{eta-0}, we also have
 \begin{corollary}\label{nodal-unicity}
 Let $\varphi$ be a real analytic function on an interval $I$. If $\dot{\varphi}\neq 0$, then for any prescribed real analytic regular curve $\alpha: I_\varepsilon\rightarrow\R^3$ with  $\langle\alpha(0),\vec{e}_3\rangle\in I$, there exists a unique $\phie3$-minimal surface passing through $\alpha(t)$, $|t|<\varepsilon'$, with $\varepsilon'$ small enough,  such that, along $\alpha$,  $H =0$ and $K\neq 0$.
 \end{corollary}
Take $\psi=(x,y,z)$ as in the statement of Theorem \ref{existence-psi} and consider now any solution $\psi^*=(x^*,y^*,z^*)$  of \eqref{CR-*} on $D_R$. Then $(\psi, \psi^*):D_R\rightarrow \R^3\times\mathbb{L}^3$ is a Calabi's pair of  real analytic weight $\varphi$.
Moreover, from \eqref{CR-*}, Proposition \ref{confpara}, \eqref{cauchy-psi}, Proposition \ref{canonical-parametrization} and Theorem \ref{existence-psi}, the singular set $\{v=0\}$ of $\psi^*$ consists only of non-degenerate singular points and the following expressions hold:
\begin{align*}
\psi^*_u(u,0)&= A^*(u) (\cos u,  -\sin u, 1),\\
\psi^*_v(u,0) &= B^*(u)( \cos u, - \sin u, 1),
\end{align*}
where $A^*,B^*\in{\cal C}^\omega(I_R)$ are as in Proposition \ref{canonical-parametrization}.

Up to an appropriate translation, we can also  assume that for any fixed point $(p_1^*,p_2^*,p_3^*)\in \mathbb{L}^3$,
$$\psi^*(u,0)= (p_1^*,p_2^*,p_3^*)+ \int_0^uA^{*}(t) (\cos t,  -\sin t, 1)dt,$$ and from \eqref{system-psi*} we have that $\psi^*=(x^*,y^*,z^*)$ is the unique solution to the following Cauchy problem,
\begin{equation}\label{cauchy-psi*}
\left\{
\begin{array}{l}
\Delta \psi^* =  \dot{\varphi^*}(z^*) (\langle \nabla x^*,\nabla z^*\rangle,\langle \nabla y^*,\nabla z^*\rangle, |\nabla x^*|^2+ |\nabla y^*|^2+ |\nabla z^*|^2),\vspace{.1in}\\
\psi^*(u,0)  = p^* +\displaystyle  \int_0^uA^*(t)( \cos t,  -\sin t, 1) dt\vspace{.1in}\\
\psi^*_v(u,0) = B^*(u) ( \cos u, -\sin u, 1),
\end{array}
\right.
\end{equation}
where $\varphi^*(z^*)=\varphi(z)$, $z^* =p^*_3 + \displaystyle \int_{p_3}^z e^{\varphi(t)} dt$.
As a consequence, we have 
\begin{theorem}\label{existence-psi*} If $\varphi^*$ is real analytic on an interval $I^*$ and $p^*=(p_1^*,p_2^*,p_3^*)\in \mathbb{L}^3$ with $p_3^* \in I^*$, then for small enough $R$ there is a unique solution $\psi^*:D_R\rightarrow \mathbb{L}^3$ to the problem \eqref{cauchy-psi*} satisfying that $\psi^*$ is a spacelike $[\varphi^*,\vec{e}_3]$-maxface whose singular set $\{v=0\}$ consists only of non-degenerate singular points. 
\end{theorem}
\begin{definition} {\rm  The curve $\mathbf{a^*}(u)= A^*(u)( \cos u,  -\sin u, 1)$ (respectively,   $\mathbf{b^*}(u)=B^*(u) ( \cos u, -\sin u, 1)$) will be called  null curve of singular directions (respectively,   limit null curve) of $\psi^*$ in the canonical conformal parametrization around the non-degenerate singular point.}\end{definition}
We already have all the necessary ingredients to prove the Theorem \ref{Th-A}, the following result below contains in a more specific way all the statements of such Theorem:
\begin{theorem}[Classification of non-degenerate singularities]\label{classification-singularities}Let $I^*\subseteq \R$  be an open interval,  $\varphi^*\in {\cal C}^\omega(I^*)$ and $p^*\in \mathbb{L}^3$ with $-\ll p^*,\vec{e}_3\gg \in I^*$. 

Consider ${\cal M}_1$ the class of spacelike $\*phie3$-maxfaces $\psi^*: \Sigma\rightarrow\mathbb{L}^3$ such that $p\in\Sigma$ is a non-degenerate singular point of $\psi^*$ with $\psi^*(p)=p^*$ and take 
${\cal M}_2(\varepsilon)=\{ (A^*,B^*)\in {\cal C}^\omega(I_\varepsilon)\times{\cal C}^\omega(I_\varepsilon) | \ \ A^{*2 }+ B^{*2} \neq 0\}$. Then, for $\varepsilon$ small enough,  the map $\Upsilon$ that sends $$\Upsilon: \psi^*\in {\cal M}_1\longrightarrow (-\ll \mathbf{a^*},\vec{e}_3\gg, -\ll \mathbf{b^*},\vec{e}_3\gg)\in {\cal M}_2(\varepsilon),$$ defines a bijective correspondence between ${\cal M}_1$ and ${\cal M}_2(\varepsilon)$, where $\mathbf{a^*}$ and $\mathbf{b^*}$ are, respectively,  the null curve of singular directions and the limit null curve of $\psi^*$  in the canonical conformal parametrization around $p$.
\end{theorem}
\begin{proof}
It is clear that $\Upsilon$ is well-defined. Let us see now that it is bijective.

Surjectivity follows from Theorem \ref{existence-psi*}. Indeed,  if  $(A^*,B^*)\in {\cal M}_2(\varepsilon)$, then from Theorem \ref{existence-psi*} the solution to the problem \eqref{cauchy-psi*} gives a spacelike $\*phie3$-maxface $\psi^*:D_R\rightarrow\mathbb{L}^3$ with a non-degenerate singular point at the origin, satisfying  $\psi^*(0,0)=p^*$ and $\Upsilon(\psi^*)= (A^*,B^*)$.

To conclude the proof, we prove the injectivity. Let $\psi^*_1,\psi^*_2\in {\cal M}_1$ such that $\Upsilon(\psi^*_1)=\Upsilon(\psi^*_2)$ and consider $\psi^*_1(u,v)$, $\psi_2^*(u,v)$, $(u,v)\in D_\varepsilon$ their canonical conformal parametrizations around $p$, then on $D_\varepsilon$,  $\psi^*_1$ and  $\psi^*_2$ are solutions of the same Cauchy problem \eqref{cauchy-psi*} and,  by the uniqueness in \eqref{cauchy-psi*}, we have that $\psi^*_1=\psi^*_2$ on $D_\varepsilon$.
\end{proof}
 \begin{figure}[h]
\begin{center}
\includegraphics[width=.40\textwidth]{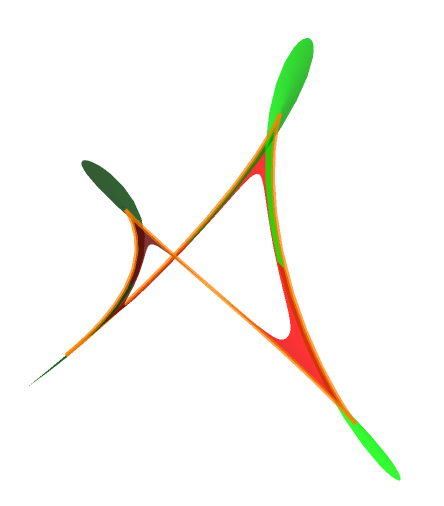} \end{center}
\caption{ Solution to \eqref{cauchy-psi*} when $\dot{\varphi}=1/z^*$,  $A^*(u)=\sin 2u$, $ B^*(u) = \cos 2u)$  }
\end{figure}
From Theorem \ref{theorem-criteria},  Proposition \ref{canonical-parametrization} and Theorem \ref{classification-singularities} we obtain
\begin{proposition} \label{type-singularity}Let $\psi^*\in {\cal M}_2$ such that $\Upsilon(\psi^*)=(A^*,B^*)$. Then
\begin{enumerate}[a)]
\item $B^* \neq 0$ if and only if $\psi^*$ is a front, which is equivalent to the limit null curve being spacelike. In this case,
\begin{enumerate}[i)]
\item  $A^*(0) \neq 0$ if and only if $p$ is a cuspidal edge.
\item $A^*(0) = 0$ and $(A^*)'(0) \neq 0$ if and only if $p$ is a swallowtail edge.
\end{enumerate}
\item We have a cuspidal cross  cap at $p$ if and only if $B^*(0)= 0$ and $(B^*)'(0)\neq 0$.
\end{enumerate}
\end{proposition} 

\begin{remark} If $A^*=0$ and $B^*$ is a non-zero constant, we obtain a rotationally symmetric conelike singularity (see Figure \ref{fig:conelike-helicoidal}-left).

If $A^*$ and $B^*$ are non-zero constants,  the examples obtained will be called helicoidal-type (see Figure \ref{fig:conelike-helicoidal}-right).
\end{remark}
 \begin{figure}[h] \label{fig:conelike-helicoidal}
\begin{center}
\includegraphics[width=.40\textwidth]{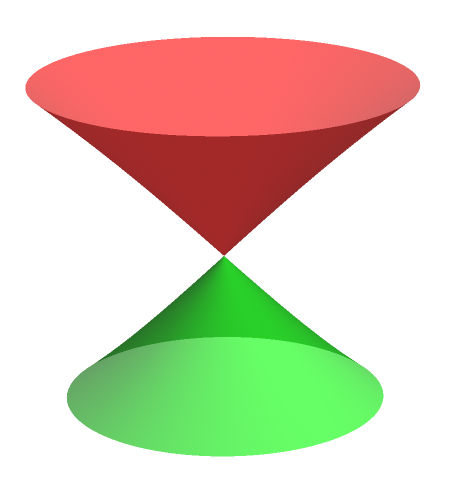} \hspace{1cm}
\includegraphics[width=.25\textwidth]{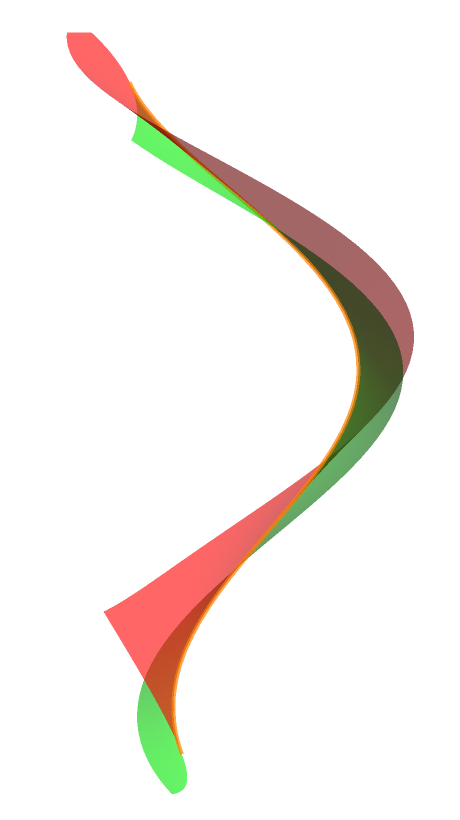} 
\end{center}
\caption{ Solution to \eqref{cauchy-psi*} with $\dot{\varphi}=1/z^*$,  left: $A^*\equiv0,B^*\equiv1$, right: $A*= B^* \equiv1$  }
\end{figure}
\section{Conelike singularities}\label{sec-5}
 In this section we study conelike singularities in the family of spacelike $\*phie3$-maximal surfaces in $\mathbb{L}^3$. To be more precise, we study non-removable isolated singularities of the following PDE:
 \begin{equation}\label{eq-graph*}
 (1-z^{*2}_{y^*})z^*_{x^*x^*}+  (1-z^{*2}_{x^*})z^*_{y^*y^*} + 2 z^*_{x^*} z^*_{y^*} z^*_{x^*y^*} = \dot{\varphi}^*(z^*) (1 - z^{*2}_{x^*}-z^{*2}_{y^*}), 
 \end{equation}
on $D^*_\rho$, where $\varphi^*\in {\cal C}^k(I^*)$, $k\geq 2$, for some open real interval $I^*$, $z^*\in {\cal C}^2(D^*_\rho)$ satisfies the ellipticity condition $1 > z^{*2}_{x^*}+z^{*2}_{y^*}$,
 $$ D^*_\rho=\{(x^*,y^*)\in \R^2\  |\  0<x^{*2} + y^{*2}<\rho^2\}.$$
 and $\rho$ is chosen such that $z^*$ is also ${\cal C}^2 $ at  the exterior boundary of $D^*_\rho$.

 The graph  of a solution $z^*$  to \eqref{eq-graph*} with a non-removable singularity at the origin is a 
$\*phie3$-maximal  surface or equivalently a maximal surface in $\R^3$ endowed with the Lorentzian metric 
$$\langle\cdot,\cdot\rangle^*= e^{-\varphi^*}\ll\cdot,\cdot\gg.$$ In consequence (see \cite{B}) $z^*$ extends continuously to the origin ($ z^*(0,0)=z_0^*$) and $$ \lim_{(x^*,y^*)\rightarrow(0,0)}\frac{|z^*(x^*,y^*)|}{\sqrt{x^{*2}+y^{*2}}}= 1,$$
which says that $z^*$ is asymptotic at the origin to the upper ($\mathbb{N}^2_+$) or lower ($\mathbb{N}^2_-$) light cone $\mathbb{N}^2=\{p\in \mathbb{L}^3\,| \, \ll p,p\gg=0\} $ of $\mathbb{L}^3$.

\

Arguing as in \cite[Lemma 1]{GJM-IHP} we can prove the following asymptotic behavior of the gradient of $z^*$,
\begin{proposition}\label{proper}
If the Gaussian curvature of the graph $(x^*,y^*,z^*(x^*,y^*))$ in $\mathbb{L}^3$ does not vanish around the singularity and  the origin is a non-removable singularity of $z^*$, then there exist a punctured neighborhood ${\cal D}^*$ of the origin in $D^*_\rho$ and $R>0$ such that 
$$ \mu : {\cal D}^*\rightarrow \mathbb{A}=\{(p^*,q^*)\, | \,  R^2<p^{*2} + q^{*2}<1\}, \quad \mu(x^*,y^*)=(z^*_{x^*},z^*_{y^*})$$
is a ${\cal C}^2$ diffeomorphism.
\end{proposition}
On  $ D^*_\rho$ we may consider the conformal structure determined by the induced metric $ds^{*2}$ of $\mathbb{L}^3$. 
We know from the uniformization theorem in \cite{S} that there exists a domain, $\Omega\subseteq\R^2=\mathbb{C}$, conformally equivalent to either a punctured disk or a certain annulus such that the change of parameters between $(x^*,y^*)$ and the conformal coordinates $(u,v)$,
$$ \Phi^*: D^*_\rho \rightarrow \Omega, \quad (x^*,y^*)\rightarrow (u(x^*,y^*),v(x^*,y^*))$$ is a ${\cal C}^2$-diffeomorphism with   positive Jacobian and such that via ${(\Phi^*)}^{-1}$ the following Beltrami system is satisfied
\begin{equation}\label{beltrami-*}
y^*_\xi = \frac{\beta^*- i \sqrt{W^*}}{\alpha^*} x^*_\xi, 
\end{equation}
where $ \xi=u+iv$, $\beta^*= z^*_{x^*}z^*_{y^*}$, $\alpha^*=1-z^{*2}_{y^*}$, $\gamma^*=1-z^{*2}_{x^*}$, $W^* = \alpha^* \gamma^* - \beta^{*2}$.

From Definition \ref{maxfaces} and Remark \ref{r-2}, the map 
$$\psi^*(u,v)=(x^*(u,v), y^*(u,v),z^*(u,v))$$ satisfies the elliptic system \eqref{system-psi*} and  if $\varphi^*$ is a solution of \eqref{punto-*} for some function $\phi^*$ then, from Theorem \ref{weierstrass-*} and Remark \ref{r-6}, the Gauss map $g$ of $\psi^*$ and $\phi^*$ satisfy the complex elliptic PDE \eqref{int-phi*}. Moreover, $\psi^*$ can be recovered,  in terms of $g$ and $\phi^*$ from the equations  \eqref{weierstrass-phi*} and \eqref{varphi-*}

In what follows, we will discuss separately the cases in which $(\Omega,d{s^*}^2)$ is conformally equivalent to either an annulus or a punctured disk.

\subsection{$(\Omega,d{s^*}^2)$ is conformally equivalent to an annulus}
In this case, we may assume that the map $\psi^*$ is well-defined on $$\psi^*:{\cal A}_\varepsilon/(2\pi\mathbb{Z})\rightarrow \mathbb{L}^3, \quad u+iv \rightarrow \psi^*(u,v),$$
where ${\cal A}_\varepsilon= \{\xi\in \mathbb{C}\, | \, 0<{\rm Im}\, \xi < \varepsilon\}$
and such that $\psi^*$ extends continuously to $\R/(2\pi\mathbb{Z})$ as $\psi^*(u,0)=(0,0,z_0^*)$. We will denote by $\widetilde{{\cal A}}_\varepsilon$ the union of $ {\cal A}_\varepsilon /(2\pi\mathbb{Z})$ with $ \R/(2\pi\mathbb{Z})$.
\begin{claim} \label{claim-1}If $\varphi^*\in {\cal C}^k(I^*)$, $k\geq 2$ (resp. $\varphi^*\in {\cal C}^\omega(I^*)$), then $\psi^*$ extends to $\widetilde{{\cal A}}_\varepsilon$ as a ${\cal C}^{k+2}$-map  (resp. ${\cal C}^{\omega}$-map).
\end{claim}\begin{proof}
The differentiability of the interior is well-known. In order to discuss the differentiability at the boundary we consider,  without loss of generality, the differentiability at the origin $(0,0)\in \widetilde{{\cal A}}_\varepsilon$.

 Since $\dot{\varphi}^*$ is bounded around $z^*_0$ and $\psi^*$ satisfies \eqref{system-psi*}, we have that 
$$ |\Delta \psi^*| \leq M |\nabla \psi^*|, \qquad \text{on ${\cal A}_\varepsilon$,}$$
for some constant $M>0$. Thus, we can apply the results in \cite{H}  and  standard potential analysis arguments to ensure that $\psi^*$ is ${\cal C}^{2,\alpha}$ in $\overline{D^+_{\varepsilon'}}$  for some $\varepsilon'<\varepsilon$ and for $0<\alpha<1$, where $D^+_{\varepsilon'}=D_{\varepsilon'}\cap  {\cal A}_\varepsilon$. Now, differentiability  follows  by a recursive process. The analyticity is a consequence 
 of \eqref{system-psi*} and \cite[Theorem 3]{Mu}. \end{proof} 
 From Claim \ref{claim-1}, the map $\psi^*$ extends to a map $$\psi^*: \widehat{{\cal A}_\varepsilon}:=\{ \xi\in \mathbb{C}\ | \ -\varepsilon <{\rm Im} \,\xi<\varepsilon\}\rightarrow \mathbb{L}^3, \quad u+ i v\rightarrow \psi^*(u,v),$$ which is a solution of \eqref{system-psi*} and satisfies
 \begin{align*}
& \psi^*(u,v) = \psi^*(u+2\pi,v), \quad v\geq 0,\\
& \psi^{*}(u,0) = (0,0,z^*_0).
\end{align*}
Moreover, from \eqref{system-psi*},
$$ \ll \psi^*_\xi,\psi^*_{\xi}\gg_{\cxi} = \dot{\varphi}^*\,z^*_{\cxi}  \,\ll \psi^*_\xi,\psi^*_{\xi}\gg,$$
that is,  $\ll \psi^*_\xi,\psi^*_{\xi}\gg$ is a pseudo-holomorphic function vanishing on $\{v\geq0\}$, but then  $\ll \psi^*_\xi,\psi^*_{\xi}\gg\equiv 0$ on $\widehat{{\cal A}_\varepsilon}$ and $\psi^*$ is a conformal map (see \cite[Theorem 6.1]{Bers}).
\begin{claim}\label{claim-2} The map $\psi^*:\widehat{{\cal A}_\varepsilon}\rightarrow \mathbb{L}^3$ is a $\*phie3$-maxface whose singular set is $\{v=0\}$ and it consists only of non-degenerate singular points.
\end{claim}
\begin{proof}
From \eqref{system-psi*} and \eqref{beltrami-*} we obtain that on $\{v>0\}$,
$$ z^*_{\xi \cxi} = \frac{\dot{\varphi}^*}{{z_x^*}^2 + {z_y^*}^2 } |z^*_\xi|^2.$$
Thus, having in mind Proposition \ref{proper} and Claim \ref{claim-1}, we have 
\begin{equation}\label{nodal-z}
|z^*_{\xi \cxi}| \leq M |z^*_\xi|^2,
\end{equation}
for some positive constant $M$. 
If $z^*_v(u_0,0)=0$, then, since $z^*_u(u_0,0)=0$ and $z^*_\xi(u_0,0)=0$, it follows from Corollary 2 in \cite{Sc} that there exists two crossing nodal curves of $z^* - z^*_0$ at $(u_0,0)$. Since one of these curves is $v=0$ that corresponds to the isolated singularity, the existence of a second nodal curve contradicts that $z^*$ is asymptotic at the origin to the upper or lower light cone. This proves that $\psi^*_v(u,0)\neq0$, $u\in\R$ and we have that  there exists $\varepsilon>0$ such that $\psi^*:\widehat{{\cal A}_\varepsilon}\rightarrow \mathbb{L}^3$ is a $\*phie3$-maxface.

From Remark \ref{r-3} we may consider  $\psi=(x,y,z): \widehat{{\cal A}_\varepsilon}\rightarrow \R^3$  a $\phie3$-minimal surface such that $(\psi,\psi^*)$ is a Calabi's pair of weight $\varphi$ (with $\varphi(z) = \varphi^*(z^*)$). Then, the singular set of $\psi^*$ is the set of points in $ \widehat{{\cal A}_\varepsilon}$ where the angle function $\eta$ of $\psi$ vanishes. 

By using \cite[Lemma 2.1]{MMT-MJM} and the conformal parameter $\xi=u+iv$ of $\psi$, we have that
$$ \eta_{\xi\cxi} + \frac{\dot{\varphi}}{2}(\eta_\xi z_{\cxi} + \eta_{\cxi}z_\xi) + ( \ddot{\varphi}|z_\xi|^2+ \frac{1}{2}(|x_\xi|^2+|y_\xi|^2+|z_\xi|^2)(H^2-2K))\eta =0,$$
and then, as $\psi$ is an immersion and $\varphi$ is ${\cal C}^k$, $k\geq2$,  we obtain that  at a neighborhood of $(u,0)$ 
\begin{equation} \label{eta-nd}| \eta_{\xi\cxi} | \leq M  (|\eta| + |\eta_\xi|), \end{equation}
for some positive constant $M$. Since $\eta(u,0)=0$, we conclude that  $\eta_v(u,0)\neq 0$ otherwise, from \cite[Corollary 2] {Sc} and \eqref{eta-nd} either $\eta\equiv 0$ or there are two nodal curves of $\eta$ through $(u,0)$, in both cases we obtain a contradiction because $\psi^*$ is a spacelike on $\{v>0\}$. Thus, for $\varepsilon$ small enough, $\{v=0\}$ is the singular set of $\psi^*$ and it consists only of non-degenerate singular points. 
\end{proof}

\begin{proposition}\label{jordan-curve} Let $(\psi,\psi^*): \widehat{{\cal A}_\varepsilon}\rightarrow \R^3 \times \mathbb{L}^3$ be the  Calabi's pair in Claim \ref{claim-2}. Then  the curve 
$$\widetilde{\gamma}:\R/(2\pi \mathbb{Z})\rightarrow \R^2\subset\R^3, \qquad u\rightarrow  \psi_u(u,0)$$ is a planar regular Jordan curve bounding a planar star-shaped domain at the  origin $\mathbf{0}\in \R^2\subset\R^3$.
\end{proposition}
\begin{proof} Since $(\psi,\psi^*)$ is a Calabi's pair of weight $\varphi$ and $z_v(u,0)=e^{-\varphi^*(z^*_0)} z^*_v(u,0)\neq 0$, then we can  write locally as
\begin{align*}
&\psi^*_u(u,0)= (0,0,0),\\
&\psi^*_v(u,0)=  e^{\varphi^*(z^*_0)} (b_1,b_2,b_3)(u), \quad b_1^2 + b^2_2 = b_3^2, \quad  b_3\neq 0,\\
&\psi_u(u,0)= (b_2,-b_1,0)(u),\\
&\psi_v(u,0)=(0,0,b_3)(u), \\
& g(u,0) = \frac{b_1-i b_2}{b_3}(u),
\end{align*}
where $g$ is the Gauss map of $\psi$ and $b_3$ is $2\pi$-periodic. 

But from \cite[Theorem 1]{kenmotsu} and Claim \ref{claim-2}  the non-degenerate singular condition means that $g_{\cxi}(u,0)=0$ and $|g'(u,0)|\neq 0$. Thus
$$0\neq |g'(u,0)|^2 = \frac{(b_1'b_3 - b_3'b_1)^2+(b_2'b_3 - b_3'b_2)^2}{b_3^2} = {b'}_1^2 + {b'}_2^2 -{b'}_3^2$$ and since $b_1^2+b^2_2=b_3^2$ we conclude  that $\widetilde{\gamma} $ is regular.   The embeddedness of $\widetilde{\gamma} $ follows because  $\widetilde{\gamma} = b_3 \, (i \overline{g},0)$ with $b_3\neq0$  and we know from Proposition \ref{proper} that  $g_0:\R/(2\pi\mathbb{Z})\rightarrow \mathbb{S}_1 $, $g_0(u)=g(u,0)$ is injective.
\end{proof}
\begin{definition}\label{singular-directions} Let $(\psi,\psi^*)$ be a Calabi's pair as  in Proposition \ref{jordan-curve}. We will say that $(\psi,\psi^*)$  is the Calabi's pair of  $z^*$. The curve $\widetilde{\gamma}:\R/(2\pi \mathbb{Z})\rightarrow \R^2\subset\R^3$, $$\widetilde{\gamma}=\psi_u(u,0)=\langle \psi_v(u,0),\vec{e}_3\rangle  (\overline{g}(u,0),0), $$ will be called the   direction curve   associated to the Calabi's pair of  the graph $z^*$  in the conformal parameter $(u,v)$.
\end{definition}
\begin{remark}\label{canonical-z*} Since the Gauss map $g$ of $\psi$ is real analytic in $\widehat{{\cal A}_\varepsilon}$ and $g$ coincides with the Gauss map of $z^*$ (see Remark \ref{r-6}), then $g(u,0):\R/(2\pi\mathbb{Z})\rightarrow \mathbb{S}^1$ is real analytic, $2\pi$-periodic and $|g'(u,0)|\neq0$, for every $u$. In consequence we may prove, by holomorphic extension, the existence of a unique conformal parameter $(u,v)$ in some $\widehat{{\cal A}_\delta}$ such that $g(u,0)= \cos u + i \sin u$. From Proposition \ref{canonical-parametrization} and Definition \ref{def-canonical}, this conformal parametrization will be called the conformal parametrization of the graph $z^*$.
\end{remark}
\subsection{$(\Omega,d{s^*}^2)$ is conformally equivalent to a punctured disk}
Here, we will assume  that $\Omega$ is a punctured disk ${\cal D}^*$.  If we take the map $\Gamma: {\cal D}^*   \rightarrow \R^2$ given by $\Gamma(u,v) = (z^*_{x^*}(u,v),z^*_{y^*}(u,v))$, then from \eqref{eq-graph*} and \eqref{beltrami-*}, we have
\begin{align}\label{sobolev} \Gamma_u^2 + \Gamma_v^2 &= (x^*_u y^*_v-x^*_v y*_u)\left(\dot{\varphi}^* \sqrt{W^*}(r^* + t^*)- \frac{1 + W^*}{\sqrt{W^*}} (r^* t^* - s^{*2})\right),
\end{align}
where $r^*=z^*_{x^*x^*}$, $t^*=z^*_{y^*y^*}$,  $s^*= z^{*}_{x^*y^*}$ and $W^*=1-z^{*2}_{x^*}-z^{*2}_{y^*}$.

Let  $\mathscr{A}_\varepsilon=\{(x^*,y^*)\ | \ \varepsilon < x^{*2} + y^{*2}<\rho^2\}$, then by application of Stokes' theorem to $\dot{\varphi}^* \sqrt{W^*}(z^*_{x^*},z^*_{y^*})$ and from \eqref{eq-graph*}, we have
\begin{align}
&\int_{\mathscr{A}_\varepsilon} \dot{\varphi}^* \sqrt{W^*}(r^* + t^*) dx^* dy^* = \label{1-sobolev}\\ &=\int_{\partial\mathscr{A}_\varepsilon}\dot{\varphi}^* \sqrt{W^*}<(z^*_{x^*},z^*_{y^*}),\mathfrak{n}> ds 
- \int_{\mathscr{A}_\varepsilon} \ddot{\varphi}^* \sqrt{W^*}(1 - W^*) dx^* dy^* + \nonumber\\
&+ \int_{\mathscr{A}_\varepsilon} \frac{\dot{\varphi}^*}{\sqrt{W^*}}( r^* z^{*2}_{x^*} + t^*z^{*2}_{y^*} + 2 z^*_{x^*} z^*_{y^*} s^*) dx^* dy^* \leq \nonumber\\
&\leq M_0 +  \int_{\mathscr{A}_\varepsilon} \dot{\varphi}^{*2} \sqrt{W^*} dx^* dy^*  - \int_{\mathscr{A}_\varepsilon} \dot{\varphi}^* \sqrt{W^*}(r^* + t^*) dx^* dy^* \leq \nonumber\\
&\leq M_1  - \int_{\mathscr{A}_\varepsilon} \dot{\varphi}^* \sqrt{W^*}(r^* + t^*) dx^* dy^*, \nonumber
\end{align}
for some constants $M_0>0$ and $M_1>0$, where $\mathfrak{n}$ stands the exterior normal to $\partial \mathscr{A}_\varepsilon$ and  we have used that $W^*$,  $\dot{\varphi}^*$, $\ddot{\varphi}^*$ are bounded and  $|<(z^*_{x^*},z^*_{y^*}),\mathfrak{n}> |\leq1$.

 If the Gaussian curvature of the graph $z^*(x^*,y^*)$ does not vanish around the singularity we can apply Proposition \ref{proper} to prove that 
 \begin{align}
 & \left|\int_{\mu^{-1}(\mathbb{A})}\frac{1 + W^*}{\sqrt{W^*}} (r^* t^* - s^{*2})) dx^*dy^*\right| \leq \int_{\mathbb{A}} \frac{2 - p^{*2} - q^{*2}}{\sqrt{1- p^{*2} - q^{*2}}} dp^* dq^* = \label{2-sobolev}\\
 &= \frac{2\pi }{3}(4-R^2)\sqrt{1-R^2}<+\infty,\nonumber
 \end{align}
 where $\mathbb{A}$ and $\mu$ are as in Proposition \ref{proper}.
 
 From \eqref{sobolev}, \eqref{1-sobolev} and \eqref{2-sobolev}, we have that $\Gamma \in {\cal W}^{1,2}({\cal D}^*)$ and we can apply Lemma 3.1 in \cite{Cu} to prove that there exists $r_n\searrow 0$ such that 
 $$\lim_{n\rightarrow +\infty} \int_{\partial{\cal D}_{r_n}}|d\Gamma(u,v)| =0,$$
 that is, there exists a subsequence of $\{\Gamma(\partial{\cal D}_{r_n})\}$ collapsing into a fixed point of the closed unit disk, which contradicts Proposition \ref{proper}.
 
As consequence we have the following removability result,
\begin{theorem}\label{removable} Let  $z^*$ be a solution of \eqref{eq-graph*} on $D^*_\rho$  such that the Gaussian curvature of its graph in $\mathbb{L}^3$  does not vanish around the singularity. If $(D^*_\rho,ds^{*2})$ is conformally equivalent to a punctured disk, then $z^*$ has a removable singularity at the puncture. 
\end{theorem}
\subsection{Existence and classification of conelike singularities}
Let  $\widetilde{\gamma}$ be a real analytic, regular Jordan planar curve bounding a planar star-shaped domain $\Omega_{\widetilde{\gamma}}$ at the origin $\mathbf{0}\in \Omega_{\widetilde{\gamma}}$. Then, we may assume that $$\widetilde{\gamma}: \R/(2\pi\mathbb{Z}) \rightarrow \R^2\subset \R^3, \qquad \widetilde{\gamma}(u) = b(u)(\sin u, \cos u,0),$$ where $b$ is a  real analytic, $2\pi$-periodic and non-vanishing function.

Consider $\varphi \in {\cal C}^\omega(I)$ in some real interval $I$ and take for $\varepsilon$ small enough the unique real analytic solution $$\psi=(x,y,z):\widehat{{\cal A}}_\varepsilon=\{\xi=u+iv\ | \ -\varepsilon < {\rm Im}\ \xi < \varepsilon\} \rightarrow \R^3, \quad \xi\rightarrow \psi(u,v),$$ to the Cauchy problem \eqref{cauchy-psi} when $A\equiv 0$ and $B(u)=b(u)$. Then, from Theorem \ref{existence-psi}, $\psi$ is a conformally immersed $\phie3$-minimal immersion whose angle function $\eta$ satisfies $\eta(u,v)\neq 0$ if $v\neq0$, $\eta(u,0)=0$ and $\eta_v(u,0)=1$ for all $u$, that is, $\psi$ is, locally, a vertical $\phie3$-minimal vertical graph on $\widehat{{\cal A}}_\varepsilon\cap \{v>0\}$.

In particular if $(\psi,\psi^*): \widehat{{\cal A}}_\varepsilon \rightarrow\R^3\times \mathbb{L}^3$ is a Calabi's pair of weight $\varphi$, then $\psi^*=(x^*,y^*,z^*)$ is the unique solution  to the Cauchy problem \eqref{cauchy-psi*} with $A^*\equiv0$ and $B^{*}(u)=e^{\varphi(z_0)}b(u)$, $z_0=\psi(0,0)$. Since the initial data of \eqref{cauchy-psi*} are $2\pi$-periodic, by the unicity of the solution, $\psi^*$ induces a well-defined  real analytic map $$\psi^*:\widehat{{\cal A}}_\varepsilon /(2\pi\mathbb{Z})\rightarrow\mathbb{L}^3, \quad (u,v) \rightarrow (x^*,y^*,z^*)(u,v)$$ which is a  $\*phie3$-maxface with singular set $\{v=0\}$ satisfying that, locally,  is a spacelike vertical $\*phie3$-maximal graph on ${\cal A}_\varepsilon /(2\pi\mathbb{Z})$, 
\begin{align*}
&\psi^*(u,0) = (p^*_1,p^*_2,p^*_3)\in \mathbb{L}^3,\\
&\psi^*_v(u,0)= e^{\varphi(z_0)}b(u) (\cos u,-\sin u,1),
\end{align*}
 and where $\varphi^*(z^*)=\varphi(z)$, $z^* =p^*_3 + \displaystyle \int_{p_3}^z e^{\varphi(t)} dt$. 
 \begin{proposition}\label{graph-phi*} The map $\psi^*:{\cal A}_\varepsilon /(2\pi\mathbb{Z})\rightarrow\mathbb{L}^3$ is globally a  spacelike $\*phie3$-maximal vertical graph of a real analytic function $z^*=z^*(x^*,y^*)$ on some punctured disk with a non-removable singularity at the puncture.
 \end{proposition} 
 \begin{proof}It is clear that
 $$\lim_{v\rightarrow 0} \frac{(z^*(u,v)-p^*_3)^2}{(x^*(u,v)-p^*_1)^2+(y^*(u,v)-p^*_2)^2}= \frac{z^{*2}_v}{x^{*2}_v + y^{*2}_v}(u,0) =1,$$
 and we may assume that $\psi^*(u,v)$ is tangent  to the upper null cone $\mathbb{N}^2_+$.
 
 Since we $\psi^*$ is, locally, a vertical graph on ${\cal A}_\varepsilon/(2\pi\mathbb{Z})$, we have that the horizontal projection $\Pi:{\cal A}_\varepsilon/(2\pi\mathbb{Z})\rightarrow \R^2$, $\Pi(u,v)=(x^*(u,v),y^*(u,v))$ is a local diffeomorphism satisfying $\Pi(u,0)=(p^*_1,p^*_2)$. Using now standard topological arguments we prove that for some $\varepsilon'<\varepsilon$, $\Pi$ is a covering map from ${\cal A}_{\varepsilon'}/(2\pi\mathbb{Z})$ onto a punctured neighborhood  of $(p^*_1,p^*_2)$ and the number of sheets of $\Pi$ is the degree of $\gamma^*_\delta=\psi^*(u,v)\cap\{z^*(u,v)=\delta\}$, $\delta>0$ small enough, which is calculated as the degree of $\mu/|\mu|(\gamma^*_\delta)=g/|g|(\gamma^*_\delta)$ where $g$ is the Gauss map of $\psi^*$ and $\mu$ is as in Proposition \ref{proper}. But the Gauss map of $\psi^*$ coincides with the Gauss map of $\psi$ which extends as a real analytic map to $\widehat{{\cal A}_{\varepsilon'}}$, in particular on the real axis $g(u,0)= \cos u + i \sin u$ and then the degree of $g/|g|(\gamma^*_\delta)$ is one, which concludes the proof. 
 \end{proof}
  \begin{figure}[h]
\begin{center}
\includegraphics[width=.40\textwidth]{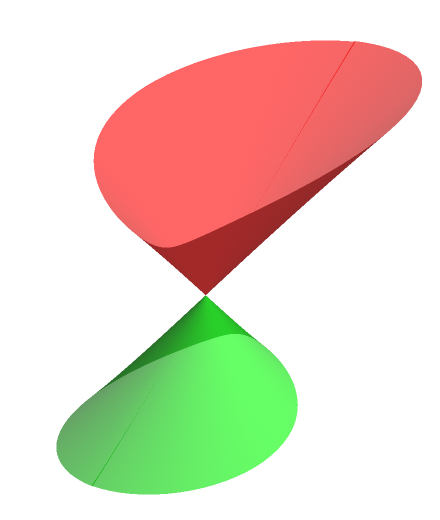} 
\end{center}
\caption{ Solution to \eqref{cauchy-psi*} with $\dot{\varphi}=1/z^*$,  $A^*\equiv0, B^*(u)= 2 + \cos u$.  }
\end{figure}
Using the results of the previous subsections, we can prove the local classification of $[\varphi^*,\vec{e}_3]$-maximal graphs with a non-removable isolated singularity in  Theorem \ref{Th-B}. The result  below states more precisely the statements established in such theorem,
 \begin{theorem}\label{classification-conelike}Let $\varphi^*\in {\cal C}^\omega(I^*)$ and $p^*=(p^*_1,p^*_2,p^*_3)\in\mathbb{L}^3$ a fixed point such that $p^*_3\in I^*$. Let ${\cal M}$ be  the class of all spacelike $\*phie3$-maximal vertical graphs in $\mathbb{L}^3$, with upwards-pointing normal, having $p^*$ as a non-removable isolated singularity and whose Gaussian curvature does not vanish near $p^*$ (we identify two graphs in ${\cal M}$ if they overlap on an open set containing $p^*$).
 
 If  $\mathscr{P}$ denotes the class of analytic regular planar Jordan curves bounding a star-shaped domain at the origin, then the map $$\varTheta:{\cal M}\rightarrow \mathscr{P}, \quad \varTheta(\Sigma)=\widetilde{\gamma},\quad \widetilde{\gamma}(u)= \psi_u(u,0),$$ that send each graph $\Sigma\in {\cal M}$ to the  direction curve, with respect to the canonical conformal parametrization, associated to the Calabi's pair of the graph $\Sigma$,  provides a one-to-one correspondence between ${\cal M}$ and $\mathscr{P}$.
 \end{theorem}
\begin{proof} We may assume that $p^*=(0,0,z_0^*)$ and it is clear from Proposition \ref{jordan-curve} and  Proposition \ref{graph-phi*} that $\varTheta$ is well-defined and  surjective. 

Now, we prove the injectivity $\varTheta$. Consider $\Sigma_1, \Sigma_2\in {\cal M}:$ and $(\psi_1,\psi_1^*)$ and $(\psi_2,\psi_2^*)$ their respective, associate Calabi's pairs defined on their canonical conformal parameters. Then, $\psi_1^*$ and $\psi_2^*$ are solutions of the same Cauchy problem \eqref{cauchy-psi*} and, by uniqueness, $ \psi_1^*(u,v)=\psi_2^*(u,v)$. That is, $\Sigma_1$ and $\Sigma_2$  overlap on a neighborhood of the singularity which finishes the proof. 
\end{proof}

\end{document}